\documentclass[reqno,12pt]{amsart}

\usepackage{color} 
\usepackage{ifpdf}
\usepackage{ifthen}
\ifpdf
    \usepackage[pdftex]{graphicx}
    \usepackage[pdftex]{hyperref}
    \hypersetup{
        unicode=false,          
        pdftoolbar=true,        
        pdfmenubar=true,        
        pdffitwindow=false,     
        pdfstartview={FitH},    
        pdftitle={MCP Article},      
        pdfauthor={Michael Holst},   
        pdfsubject={Mathematics},    
        pdfcreator={Michael Holst},  
        pdfproducer={Michael Holst}, 
        pdfkeywords={PDE, analysis, mathematical physics}, 
        pdfnewwindow=true,      
        colorlinks=true,        
        linkcolor=red,          
        citecolor=blue,         
        filecolor=magenta,      
        urlcolor=cyan           
    }

\else
    \usepackage{graphicx}
    \usepackage{pstricks}
    
    \newcommand{\href}[2]{#2}
\fi

\usepackage{times}
\usepackage{amsfonts}

\usepackage{amsmath}
\usepackage{amsthm}
\usepackage{amssymb}
\usepackage{amsbsy}
\usepackage{amscd}

\usepackage{enumerate}
\usepackage{verbatim}
\usepackage{subcaption}

\newtheorem{theorem}{Theorem}[section]
\newtheorem{corollary}[theorem]{Corollary}
\newtheorem{lemma}[theorem]{Lemma}

\newtheorem{definition}[theorem]{Definition}
\newtheorem{example}[theorem]{Example}

\numberwithin{equation}{section}  

  \newcounter{mnote}
  \let\oldmarginpar\marginpar
    \renewcommand\marginpar[1]{\-\oldmarginpar[\raggedleft\footnotesize #1]%
    {\raggedright\footnotesize #1}}

\definecolor{myblue}{rgb}{0.2,0.2,0.7}
\definecolor{mygreen}{rgb}{0,0.6,0}
\definecolor{mycyan}{rgb}{0,0.6,0.6}
\definecolor{myred}{rgb}{0.9,0.2,0.2}
\definecolor{mymagenta}{rgb}{0.9,0.2,0.9}
\definecolor{mywhite}{rgb}{1.0,1.0,1.0}
\definecolor{myblack}{rgb}{0.0,0.0,0.0}

\newcommand{\HO}[3][]{\ensuremath{H^{#1}\Omega^{#2} (#3)}}

\newcommand{\vn}[2][]{\ifthenelse{\equal{#1}{}}{\ensuremath{\|#2\|}}{\ensuremath{\|#2\|_{#1}}}} 
\newcommand{\bvn}[2][]{\ifthenelse{\equal{#1}{}}{\ensuremath{\left\|#2\right\|}}{\ensuremath{\left\|#2\right\|_{#1}}}} 
\newcommand{\mn}[2][]{\ifthenelse{\equal{#1}{}}{\ensuremath{\|#2\|}}{\ensuremath{\|#2\|_{#1}}}}

\newcommand {\aip}[2][]{\langle #2 \rangle_{#1}}

\newcommand {\baip}[2][]{\left\langle #2 \right\rangle_{#1}}

\newcommand {\keyterm}[2][]{\textbf{#2}}

\usepackage[curve]{xypic}
\newcommand{\poly} {{\mathcal P}}

\newcommand{\Lin} {\mathcal L}
\newcommand{\R} {\mathbb R}

\newcommand{\Div}{\operatorname{div}}

\newcommand{\cancel}[1]{}
\newcommand{\raw}{\rightarrow}

\newcommand{\Tau}{\mathcal T}

\newcommand{\veps}{\varepsilon}

\newcommand\dist{\operatorname{dist}}

\begin{document}

\title[FEEC for Evolution Problems on Manifolds]   
      {Finite Element Exterior Calculus for Parabolic Evolution Problems On Riemannian Hypersurfaces}

\author[M. Holst]{Michael Holst}
\email{mholst@math.ucsd.edu}

\author[C. Tiee]{Christopher Tiee}
\email{ctiee@ccom.ucsd.edu}

\address{Department of Mathematics\\
         University of California San Diego\\ 
         La Jolla CA 92093}

\thanks{MH was supported in part by
        NSF Awards~1217175, 1262982, and 1318480.
        CT was supported in part by
        NSF Award~1217175}
\date{\today}


%
\begin{abstract}
Over the last ten years, the Finite Element Exterior Calculus (FEEC) has
been developed as a general framework for linear mixed variational problems, 
their numerical approximation by mixed methods, and their error analysis. 
The basic approach in FEEC, pioneered by Arnold, Falk, and Winther in two
seminal articles in 2006 and 2010, interprets these problems in the
setting of {\em Hilbert complexes}, leading to a more general and complete 
understanding.
Over the last five years, the FEEC framework has been extended to a broader
set of problems.
One such extension, due to Holst and Stern in 2012, was to problems 
with {\em variational crimes}, allowing for the analysis and numerical 
approximation of linear and geometric elliptic partial differential 
equations on Riemannian manifolds of arbitrary spatial dimension.
Their results substantially generalize the existing surface finite element 
approximation theory in several respects.
In 2014, Gillette, Holst, and Zhu extended the FEEC in another direction, 
namely to parabolic and hyperbolic evolution systems by combining the FEEC
framework for elliptic operators with classical approaches for parabolic
and hyperbolic operators, by viewing solutions to the evolution problem
as lying in Bochner spaces (spaces of Banach-space valued parametrized curves).
Related work on developing an FEEC theory for parabolic evolution problems 
has also been done independently by Arnold and Chen.
In this article, we extend the work of Gillette-Holst-Zhu and Arnold-Chen
to evolution problems on Riemannian manifolds, through the use
of framework developed by Holst and Stern for analyzing variational crimes.
We establish \emph{a priori} error estimates that reduce to the results 
from earlier work in the flat (non-criminal) setting. 
Some numerical examples are also presented. 
\end{abstract}

\maketitle

\clearpage

\tableofcontents


%
\newcommand{\f}[1]{\mathfrak{#1}}
\newcommand{\cl}[1]{\overline{#1}}
\newcommand{\e}[1]{\mathcal{#1}}
\newcommand{\rst}[2]{\left.#1\right|_{#2}}
\section{Introduction}
Arnold, Falk, and Winther~\cite{AFW2006,AFW2010} introduced the Finite Element Exterior Calculus (FEEC) as a general framework for linear mixed variational problems, their numerical approximation by mixed methods, and their error analysis. They recast these problems using the ideas and tools of {\em Hilbert complexes}, leading to a more complete understanding. Subsequently, Holst and Stern~\cite{HoSt10a}
extended the Arnold--Falk--Winther framework to include {\em variational
crimes}, allowing for the analysis and numerical approximation of linear 
and geometric elliptic partial differential equations on 
Riemannian manifolds of arbitrary spatial dimension, generalizing the 
existing surface finite element approximation theory in several 
directions. Gillette, Holst, and Zhu~\cite{GiHo11a} extended the FEEC in another direction, namely
to parabolic and hyperbolic evolution systems by combining recent work on the FEEC for elliptic problems 
with a classical approach of Thom\'ee~\cite{T2006} to solving evolution problems using semi-discrete
finite element methods, by viewing solutions to the evolution problem
as lying in Bochner spaces (spaces of Banach-space valued parametrized curves). Arnold and Chen~\cite{AC2012} independently developed related work, for generalized Hodge Laplacian parabolic problems for differential forms of arbitrary degree, and Holst, Mihalik, and Szypowski~\cite{HMS13a} consider similar work in adaptive finite element methods. In this article, we aim to combine the approaches of the above articles, extending the work of Gillette, Holst, and Zhu~\cite{GiHo11a} and Arnold and Chen~\cite{AC2012} to parabolic evolution problems on Riemannian manifolds by using the framework of Holst and Stern~\cite{HoSt10a}.

\subsection{The Hodge heat equation and its mixed form}
We now introduce our problem by some concrete motivation. We consider an evolution equation for differential forms on a manifold. Then we rephrase it as a mixed problem as an intermediate step toward semidiscretization using mixed finite element methods. We then see how this allows us to leverage existing \emph{a priori} error estimates for parabolic problems, and see how it fits in the framework of Hilbert complexes.

Let $M$ be a compact oriented Riemannian $n$-manifold embedded in $\R^{n+1}$. The \keyterm{Hodge heat equation} is to find time-dependent $k$-form $u : M \times  [0,T] \to \Lambda^k(M)$ such that
\begin{equation}
\label{eq:par-cnts}
\begin{tabular}{rlll}
$u_t - \Delta u = u_{t} + (\delta d + d\delta)u = f$ & in $M$, & for $t>0$ \\
$u(\cdot,0)=g$ & in $M$.
\end{tabular}
\end{equation}
where $g$ is an initial $k$-form, and $f$, a possibly time-dependent $k$-form, is a source term. Note that no boundary conditions are needed for manifolds without boundary. This is the problem studied by Arnold and Chen~\cite{AC2012}, and in the case $k=n$, one of the problems studied by Gillette, Holst, and Zhu~\cite{GiHo11a}, building upon work in special cases for domains in $\R^{2}$ and $\R^{3}$ by Johnson and Thom\'ee~\cite{JT1981,T2006}

For the stability of the numerical approximations with the methods of \cite{HoSt10b} and \cite{AFW2010}, we recast the problem in mixed form, converting the problem into a system of differential equations. Motivating the problem by setting $\sigma = \delta u$ (recall that for the Dirichlet problem and $k=n$, $\delta$ here corresponds to the gradient in Euclidean space, and is the adjoint $d$, corresponding to the \emph{negative} divergence), and taking the adjoint, we have
\begin{equation}
\label{eq:par-mixedweak}
\begin{tabular}{rllll}
$\aip{\sigma,\omega} - \aip{u,d \omega}$ & $= 0,$ & $\forall~\omega\in H\Omega^{k-1}(M),$ & $t>0$, \\
$\aip{u_{t},\varphi} +\aip{d\sigma,\varphi}+ \aip{du,d\varphi}$ & $= \aip{f,\varphi},$ & $\forall~ \varphi\in H\Omega^{k}(M)$ & $t>0$.\\
 $u(0)$ &$=g$.
\end{tabular}
\end{equation}
Unlike the elliptic case, we do not have to explicitly account for harmonic forms in the formulation of the equations themselves, but they will definitely play a critical role in our analysis and bring new results not apparent in the $k=n$ case.

\subsection{Semidiscretization of the equation}
In order to analyze the numerical approximation, we semidiscretize our problem in space. In our case, we shall assume, following \cite{HoSt10a}, that we have a family of approximating surfaces $M_h$ to the hypersurface $M$, given as the zero level set of some signed distance function, all contained in a tubular neighborhood $U$ of $M$, and a projection $a : M_h \to M$ along the surface normal (of $M$). The surfaces may be a triangulations, i.e., piecewise linear (studied by Dziuk and Demlow in \cite{Dziuk88,DeDz06}), or piecewise polynomial (obtained by Lagrange interpolation over a triangulation of the projection $a$, as later studied by Demlow in \cite{Demlow2009}). We pull forms on $M_h$ to $M$ back via the inverse of the normal projection, which furnishes injective morphisms $i^k_h : \Lambda^k_h \hookrightarrow H\Omega^{k}(M)$ as required by the theory in \cite{HoSt10a}, which we shall review in Section \ref{sec:feec} below. Finally, we need a family of linear projections $\Pi^k_h :  H\Omega^{k}(M) \to \Lambda^{k}_h$ such that $\Pi_h \circ i_h = \operatorname{id}$  which allow us to interpolate given data into the chosen finite element spaces---this is necessary because some of the more obvious, natural seeming choices of operators, such as $i_{h}^{*}$, can be difficult to compute (nevertheless, $i_{h}^{*}$ will still be useful theoretically).

We now can formulate the semidiscrete problem: we seek a solution $(\sigma_h,u_{h})\in H_h\times S_h\subseteq H\Omega^{k-1} \times H\Omega^{k}$ such that
\begin{equation}
\label{eq:par-semidisc-first}
\begin{tabular}{rllll}
$\aip[h]{\sigma_h,\omega_h} - \aip[h]{u_h,d\omega_h}$ & $= 0,$ & $\forall~\omega_h\in H_h,$  & $t>0$\\
$\aip[h]{u_{h,t},\varphi_h} +\aip[h]{d\sigma_h,\varphi_h}+ \aip[h]{du_{h},d\varphi_{h}}$ & $= \aip[h]{\Pi_{h}f,\varphi_h},$ & $\forall~ \varphi_h\in S_h$ & $t>0$\\
 $u_h(0)$ &$=g_h$.
\end{tabular}
\end{equation}
We shall describe how to define $g_h \in S_h$ shortly; it is to be some suitable interpolation of $g$. As $S_{h}$ and $H_{h}$ are finite-dimensional spaces, we can reduce this to a system of ODEs in Euclidean space by choosing bases $(\psi_{i})$ for $S_{h}$ and $(\phi_{k})$ for $H_{h}$; expanding the unknowns $\sigma_{h} = \sum_{i} \Sigma^{i}(t) \psi_{i}$ and $u_{h} = \sum_{k} U^{k}(t) \phi_{k}$; substituting these basis functions as test functions to form matrices $A_{k\ell} =\aip{\phi_{k},\phi_{\ell}}$, $B_{ik} = \aip{d\psi_{i},\phi_{k}}$, and $D_{ij} =\aip{\psi_{i},\psi_{j}}$; and finally forming the vectors for the load data $F$ defined by $F_{k} = \aip{F,\phi_{k}}$, and initial condition $G$ defined by $g_{h} = \sum G^{k} \phi_{k}$. We thus arrive at the matrix equations for the unknown, time-dependent coefficient vectors $\Sigma$ and $U$:
\begin{align*}
  D\Sigma -B^TU&= 0,\\
AU_t +B\Sigma  + K U&= F,
\;\;\text{for $t>0$} \\
U(0) & =G.
\end{align*}

The matrices $A$ and $D$ are positive definite, hence invertible. Substituting $\Sigma = D^{-1} B^{T} U$, we have the system of ODEs 
\[AU_t + (BD^{-1}B^T + K)U = F,\;\;\text{for $t>0$,\;\; $U(0)= G$},\]
which has a unique solution by the usual ODE theory. For purposes of actually numerically integrating the ODE, namely, discretizing fully in space and time, it is better not to use the above formulation, because it can lead to dense matrices. Computationally, this is due to the explicit presence of an inverse, $D^{-1}$, not directly multiplying the variable; conceptually, this is actually a statement about the discrete adjoint to the codifferential $d^{*}_{h}$ generally having global support even if the finite element functions are only locally supported \cite{AC2012}. Instead, we differentiate the first equation with respect to time, getting $D\Sigma_{t} - B^{T} U_{t} = 0$, which leads to the block system
\begin{equation}\label{eq:discrete-evolution-system}
\frac{d}{dt}\begin{pmatrix}
D &-B^{T}\\
0 & A
\end{pmatrix}
\begin{pmatrix}
\Sigma \\ U
\end{pmatrix} = \begin{pmatrix}0 & 0 \\
-B & -K
\end{pmatrix}\begin{pmatrix}
\Sigma \\ U
\end{pmatrix} + \begin{pmatrix} 0\\ F
\end{pmatrix}
\end{equation}
which is still well-defined ODE for $\Sigma$ and $U$, as the invertible matrices $A$ and $D$ appear on the diagonal. This differentiated equation also plays a role in the showing that the continuous problem is well-posed.

These equations differ from\ those studied by Gillette, Holst, and Zhu \cite{GiHo11a}, Arnold and Chen \cite{AC2012}, and Thom\'ee \cite{T2006} by the choice of finite element spaces---here we are assuming them to be in some Sobolev space of differential forms on manifolds (or in a triangulated mesh in a tubular neighborhood) rather than subsets of Euclidean space. This suggests that we should try to gather these commonalities, examine what happens in abstract Hilbert complexes, and see how general a form of error estimate we can get this way. 
\subsection{Error analysis}
The general idea of the method of Thom\'ee~\cite{T2006} is to compare the semidiscrete solution to an \keyterm{elliptic projection} of the data, a method first explored by Wheeler~\cite{W1973}. If we assume that we already have a solution $u$ to the continuous problem, then for each fixed time $t$, $u(t)$ can be considered as trivially solving an elliptic equation with data $-\Delta u(t)$. Thus, using the methods developed in \cite{AFW2010}, we consider the discrete solution $\tilde u_{h}$ for $u$ in this elliptic problem (namely, applying the discrete solution operator $T_{h}$ to $-\Delta u(t)$). This may be compared to the true solution (at each fixed time) using the error estimates in \cite{AFW2010}. What remains is to compare the semidiscrete solution $u_{h}$ (as defined by the ODEs \eqref{eq:par-semidisc-first} above) to the elliptic projection, so that we have the full error estimate by the triangle inequality. Thom\'ee derives the following estimates, for finite elements in the plane ($n=2$) of top-degree forms ($k=2$, there represented by a scalar proxy), for $g_h$ the elliptic projection of the initial condition $g$ and $t\geq 0$:
\begin{align}
\vn[L^2]{u_h(t)-u(t)} & \leq ch^2\left(\vn[H^2]{u(t)} +\int_0^t\vn[H^{2}]{u_t(s)}ds \right), \label{eq:thomee-1-u}\\
\vn[L^{2}]{\sigma_h(t)-\sigma(t)} & \leq ch^2\left(\vn[H^3]{u(t)} +\left(\int_0^t\vn[H^2]{u_t(s)}^2ds\right)^{1/2} \right).\label{eq:thomee-1-s}
\end{align}
Gillette, Holst, and Zhu \cite{GiHo11a}, and Arnold and Chen \cite{AC2012} generalize these estimates and represent them in terms of Bochner norms. These estimates describe the accumulation of error up to fixed time value $t$, assuming, of course, that the spaces finite elements are sufficiently regular to allow those estimates. The key equation that makes these estimates possible are Thom\'ee's error evolution equations: defining $\rho = \vn{\tilde u_{h}(t)-u(t)}$, $\theta = \vn{u_{h}(t) - \tilde u_{h}(t)}$, and $\varepsilon = \vn{\sigma_{h}(t) - \tilde \sigma_{h}(t)}$, we have
\begin{align*}
\aip{\theta_{t},\phi_{h}}-\aip{\Div \varepsilon(t),\phi_{h}} &= -\aip{\rho_{t},\phi_{h}}\\
\aip{\varepsilon,\omega_{h}} +\aip{\theta,\Div \omega_{h}}&= 0.
\end{align*}
These are used to derive certain differential inequalities and make Gr\"onwall-type estimates. In this chapter, we examine the above error equations and place them in a more abstract framework. We use Bochner spaces (also used by \cite{GiHo11a}) to describe time evolution in Hilbert complexes, building on their successful use in elliptic problems. We investigate Thom\'ee's method in this framework to gain further insight into how finite element error estimates evolve in time.

\subsection{Summary of the paper}
The remainder of this paper is structured as follows.
In Section~\ref{sec:feec}, we review the finite element exterior calculus (FEEC) and the variational crimes framework of Holst and Stern~\cite{HoSt10a}. We prove some extensions in order to account for problems with prescribed harmonic forms; this is what allows the elliptic projection to work in the case where harmonic forms are present.
In Section~\ref{sec:abs-evol}, we formulate abstract parabolic problems in Bochner spaces and extend some standard results on the existence and uniqueness of strong solutions, and describe how this problem fits into that framework.
In Section~\ref{sec:results-par}, we extend the {\em a priori} error estimates for Galerkin mixed
finite element methods to parabolic problems on Hilbert complexes. Then, we relate the resuls to the problem on manifolds. The main abstract result is Theorem~\ref{eq:main-parabolic-estimate}, which uses the previous results from the FEEC framework with variational crimes, in order to understand how those error terms evolve with time. We then specialize, in Section~\ref{sec:par-eqns-riem} to parabolic equations on Riemannian manifolds, our original motivating example, and see how this generalizes the error estimates of Thom\'ee~\cite{T2006}, Gillette, Holst, and Zhu~\cite{GiHo11a}, and Holst and Stern~\cite{HoSt10a}.
In Section~\ref{sec:num-experiments}, we present a numerical experiment comparing the methods based on this mixed form to more straightforward implementations in the scalar heat equation case.


\section{Finite Element Exterior Calculus}
\label{sec:feec}
We review here the relevant results from the finite element exterior calculus (FEEC) that we will need for this paper. FEEC was introduced in Arnold, Falk and Winther~\cite{AFW2006,AFW2010} as a framework for deriving error estimates and formulating stable numerical methods for a large class of elliptic \textsc{pde}. One of the central ideas which helped unify many of these distinct methods into a structured framework has been the idea of \keyterm[Hilbert complex]{Hilbert complexes}~\cite{BrLe92}, which abstracts the essential features of the cochain complexes commonly found in exterior calculus and places them in a context where modern methods of functional analysis may be applied. This assists in formulating and solving boundary value problems, in direct analogy to how Sobolev spaces have helped provide a framework for solving such problems for functions. Arnold, Falk, and Winther~\cite{AFW2010} place numerical methods into this framework by choosing certain finite-dimensional subspaces satisfying certain compatibility and approximation properties. Holst and Stern~\cite{HoSt10a} extended this framework by considering the case in which there is an injective morphism from a finite-dimensional complex to the complex of interest, without it necessarily being inclusion. This allows the treatment of geometric \keyterm{variational crimes}~\cite{Brae07,BrSc02}, where an approximating manifold (on which it may be far easier to choose finite element spaces) no longer coincides with the actual manifold on which we seek our solution. We review the theory as detailed in \cite{HoSt10a} and refer the reader there for details.

\subsection{Hilbert Complexes}
As stated before, the essential details of differential complexes, such as the de Rham complex, are nicely captured in the notion of Hilbert complexes. This enables us to see clearly where many elements of boundary value problems come from, in particular, the Laplacian, Hodge decomposition theorem, and Poincar\'e inequality. In addition, it allows us to see how to carry these notions over to numerical approximations.
\begin{definition}[Hilbert complexes] We define a \keyterm{Hilbert complex} $(W,d)$ to be sequence of Hilbert spaces $W^{k}$ with possibly unbounded linear maps $d^{k}: V^{k}\subseteq W^{k} \to V^{k+1} \subseteq W^{k+1}$, such that each $d^{k}$ has closed graph, densely defined, and satisfies the \keyterm{cochain property} $d^{k} \circ d^{k-1} = 0$ (this is often abbreviated $d^{2} = 0$; we often omit the superscripts when the context is clear). We call each $V^{k}$ the \keyterm[domain of a linear operator]{domain} of $d^{k}$. We will often refer to elements of such Hilbert spaces as ``forms,'' being motivated by the canonical example of the de Rham complex. The Hilbert complex is called a \keyterm[Hilbert complex!closed]{closed complex} if each image space $\f B^{k} =  d^{k-1} V^{k-1}$ (called the \keyterm[coboundary!in Hilbert complexes]{$k$-coboundaries} is closed in $W^{k}$, and a \keyterm[Hilbert complex!bounded]{bounded complex} if each $d^{k}$ is in fact a bounded linear map. The most common arrangement in which one finds a bounded complex is by taking the sequence of domains $V^{k}$, the same maps $d^{k}$, but now with the \keyterm[inner product!graph inner product in Hilbert complexes]{graph inner product}
\[
\aip[V]{v,w} = \aip{v,w} + \aip{d^{k} v, d^{k} w}.
\]
for all $v,w \in V^{k}$. Unsubscripted inner products and norms will always be assumed to be the ones associated to $W^{k}$.
\end{definition}
\begin{definition}[Cocycles, Coboundaries, and Cohomology] The kernel of the map $d^{k}$ in $V^{k}$ will be called $\f Z^{k}$, the \keyterm[cocycle]{$k$-cocycles} and, as before, we have $\f B^{k} = d^{k-1} V^{k-1}$. Since $d^{k} \circ d^{k-1} = 0$, we have $\f B^{k} \subseteq \f Z^{k}$, so we have the \keyterm[cohomology!in Hilbert complexes]{$k$-cohomology} $\f Z^{k}/\f B^{k}$. The \keyterm[harmonic form!in a Hilbert complex]{harmonic space} $\f H^{k}$ is the orthogonal complement of $\f B^{k}$ in $\f Z^{k}$. This means, in general, we have an orthogonal decomposition $\f Z^{k} = \cl{\f B^{k}} \oplus \f H^{k}$, and we have that $\f H^{k}$ is isomorphic to $\f Z^{k}/ \cl{\f B^{k}}$, the \keyterm[cohomology!reduced]{reduced cohomology}, which of course corresponds to the usual cohomology for closed complexes.
\end{definition}
\begin{definition}[Dual complexes and adjoints] For a Hilbert complex $(W,d)$, we can form the \keyterm[Hilbert complex!dual complex]{dual complex} $(W^{*},d^{*})$ which consists of spaces $W^{*}_{k}  = W^{k}$, maps  $d^{*}_{k} : V^{*}_{k}\subseteq W^{*}_{k} \to V^{*}_{k-1} \subseteq W^{*}_{k-1}$ such that $d^{*}_{k+1} = (d^{k})^{*}$, the adjoint operator, that is:
\[
\aip{d^{*}_{k+1} v, w} = \aip{v, d^{k}w}.
\]
The operators $d^{*}$ decrease degree, so this is a chain complex, rather than a cochain complex; the analogous concepts to cocycles and coboundaries extend to this case and we write $\f Z^{*}_{k}$ and $\f B^{*}_{k}$ for them.
\end{definition}

\begin{definition}[Morphisms of Hilbert complexes] Let $(W,d)$ and $(W',d')$ be two Hilbert complexes. $f : W\to W'$ is called a \keyterm{morphism of Hilbert complexes} if we have a sequence of bounded linear maps $f^{k} : W^{k} \to W^{\prime k}$ such that $d^{\prime k} \circ f^{k} = f^{k+1} \circ d^{k}$ (they commute with the differentials).
\end{definition}
With the above, we can show the following \keyterm[Hodge decomposition!in Hilbert complexes]{weak Hodge decomposition}:
\begin{theorem}[Hodge Decomposition Theorem] Let $(W,d)$ be a Hilbert complex with domain complex $(V,d)$. Then we have the $W$- and $V$-orthogonal decompositions
\begin{align}\label{eqn:HC-hodge-decomp}
W^{k} &= \cl{\f B^{k}} \oplus \f H^{k} \oplus \f Z^{k\perp_{W}}\\
V^{k} &= \cl{\f B^{k}} \oplus \f H^{k} \oplus \f Z^{k\perp_{V}}.
\end{align}
where $\f Z^{k\perp_{ V}} = \f Z^{k\perp_{W}} \cap V^{k}$.
\end{theorem}
Of course, if $\f B^{k}$ is closed, then the extra closure is unnecessary, and we omit the term ``weak''. We shall simply write $\f Z^{k\perp}$ for $\f Z^{k\perp_{V}}$, which is will be the most useful orthogonal complement for our purposes. The orthogonal projections $P_{U}$ for a subspace $U$ will be in the $W$-inner product unless otherwise stated (although again, due to the two inner products coinciding on $\f Z^{k}$ and its subspaces, they may be the same). We note that by the abstract properties of adjoints (\cite[\S 3.1.2]{AFW2010}), $\f Z^{k\perp_{W}} = \cl{\f B^{*}_{k}}$, and $\f B^{k\perp_{W}} = \f Z^{*}_{k}$. Also very useful is that the $V$- and $W$-norms agree on $\f Z$ and hence on $\f B$ and $\f H$.

The following inequality is an important result crucial to the stability of our solutions to the boundary value problems as well as the numerical approximations:
\begin{theorem}[Abstract Poincar\'e Inequality] If $(V,d)$ is a closed, bounded Hilbert complex, then there exists a constant $c_{P} >0$ such that for all $v \in \f Z^{k\perp}$,
\[
\|v\|_{V} \leq c_{P} \|d^{k}v\|_{V}.
\]
\end{theorem}
In the case that $(V,d)$ is the domain complex associated to a closed Hilbert complex $(W,d)$, $(V,d)$ is again closed, and the additional graph inner product term vanishes: $\|d^{k} v\|_{V} = \|d^{k}v\|$. We now introduce the abstract version of the Hodge Laplacian and the associated problem.

\begin{definition}[Abstract Hodge Laplacian problems] We consider the operator $L = dd^{*} + d^{*}d$ on a Hilbert complex $(W,d)$, called the \keyterm[Laplacian!Abstract Hodge]{abstract Hodge Laplacian}. Its domain is $D_{L} = \{ u \in V^{k} \cap V_{k}^{*} : du \in V_{k+1}^{*}, d^{*} u \in V^{k-1}\}$, and the \keyterm[Laplacian!Abstract Hodge problem]{Hodge Laplacian problem} is to seek $u \in V^{k} \cap V_{k}$, given $f \in W^{k}$, such that
\begin{equation}\label{eqn:basic-abstract-hodge}
\aip{du,dv} + \aip{d^{*}u,d^{*}v} = \aip{f,v}
\end{equation}
for all $v \in V^{k} \cap V^{*}_{k}$. This is simply the weak form of the Laplacian and any $u \in V^{k}\cap V^{*}
_{k}$ satisfying the above is called a \keyterm[weak solution!to the Abstract Hodge problem]{weak solution}. Owing to difficulties in the approximation theory for such a problem (it is difficult to construct finite elements for the space $V^{k} \cap V_{k}^{*}$), Arnold, Falk, and Winther~\cite{AFW2010} formulated the \keyterm{mixed abstract Hodge Laplacian problem} by defining auxiliary variables $\sigma = d^{*} u$ and $p = P_{\f H} f$, the orthogonal projection of $f$ into the harmonic space, and considering a \emph{system} of equations, to seek $(\sigma,u,p) \in V^{k-1} \times V^{k} \times \f H^{k}$ such that
\begin{equation}\label{eqn:mixed-hodge-laplacian}
\begin{tabular}{rll}
$\aip{\sigma, \tau} - \aip{u, d\tau}$ & $=0$ & $\forall \tau \in V^{k-1}$\\
$\aip{d\sigma, v} + \aip{du,dv} + \aip{p,v}$ & $=\aip{f,v}$ & $\forall v\in V^{k}$\\
$\aip{u, q} $ & $=0$ & $\forall q \in \f H^{k}$.
\end{tabular}
\end{equation}

\end{definition}
The first equation is the weak form of $\sigma = d^{*}u$, the second is the main equation \eqref{eqn:basic-abstract-hodge} modified to account for a harmonic term so that a solution exists, and the third enforces uniqueness by requiring perpendicularity to the harmonic space. With these modifications, the problem is well-posed by considering the bilinear form (writing $\f X^{k} := V^{k-1} \times V^{k} \times \f H^{k}$) $B: \f X^{k} \times \f X^{k} \to \R$ defined by
\begin{equation}\label{eqn:afw-bilin-form}
B(\sigma,u,p;\tau,v,q) := \aip{\sigma,\tau} - \aip{d\tau, u} + \aip{d\sigma,v} +\aip{du,dv} + \aip{p,v} - \aip{u,q}.
\end{equation}
and linear functional $F \in (\f X^{k})^{*}$ given by $F(\tau,v,q) = \aip{f,v}$. The form $B$ is \emph{not} coercive, but rather, for a closed Hilbert complex, satisfies an \keyterm{inf-sup condition} \cite{AFW2010,Babuska.I1971}: there exists $\gamma > 0$ (the \keyterm{stability constant}) such that
\[
\inf_{(\sigma,u,p)\neq 0} \sup_{(\tau,v,q)\neq 0} \frac{B(\sigma,u,p;\tau,v,q)}{\| (\sigma,u,p) \|_{\f X} \|(\tau, v, q)\|_{\f X}} =: \gamma > 0.
\]
where we have defined a standard norm on products: $\|(\sigma,u,p)\|_{\f X} := \|\sigma\|_{V} + \|u\|_{V} + \|p\|$. This is sufficient to guarantee the well-posedness \cite{Babuska.I1971}. To summarize:
\begin{theorem}[Arnold, Falk, and Winther~\cite{AFW2010}, Theorem 3.1] The mixed variational problem \eqref{eqn:mixed-hodge-laplacian} on a closed Hilbert complex $(W,d)$ with domain $(V,d)$ is well-posed: the bilinear form $B$ satisfies the inf-sup condition with constant $\gamma$, so for any $F \in  (\f X^{k})^{*}$, there exists a unique solution $(\sigma,u,p)$ to \eqref{eqn:mixed-hodge-laplacian}, i.e., $B(\sigma,u,p;\tau,v,q) = F(\tau,v,q)$ fo all $(\tau,v,q)\in\f X^{k}$, and moreover,
\[
\|(\sigma,u,p)\|_{\f X} \leq  \gamma^{-1} \|F\|_{\f X^{*}}.
\]
The \keyterm{stability constant} $\gamma^{-1}$ depends only on the Poincar\'e constant.
\end{theorem}
Note that the general theory (e.g., \cite{Babuska.I1971,Ev98}) guarantees a unique solution exists for \emph{any} bounded linear functional $F \in (\f X^{k})^{*}$, which in this case with product spaces, means that the problem is still well-posed when there are other nonzero linear functionals on the \textsc{rhs} of \eqref{eqn:mixed-hodge-laplacian} besides $\aip{f,v}$. We shall need this result for parabolic problems, where we assume $u$ has a harmonic part ($P_{\f H} u \neq 0$).
\subsection{Approximation of Hilbert Complexes}\label{sec:approx-hilb-complex}
We now approximate solutions to the abstract mixed Hodge Laplacian problem. To do so, Arnold, Falk, and Winther~\cite{AFW2010} introduce finite-dimensional subspaces $V_{h}\subseteq V$ of the domain complex, such that the inclusion $i_{h} :V_{h} \hookrightarrow V$ is a morphism, i.e. $dV_{h}^{k} \subseteq V_{h}^{k+1}$. With the weak form \eqref{eqn:mixed-hodge-laplacian}, we formulate the Galerkin method by restricting to the subspaces:
\begin{equation}\label{eqn:mixed-hodge-laplacian-discrete}
\begin{tabular}{rll}
$\aip{\sigma_{h}, \tau} - \aip{u_{h}, d\tau}$ & $=0$ & $\forall \tau \in V_{h}^{k-1}$\\[2mm]
$\aip{d\sigma_{h}, v} + \aip{du_{h},dv} + \aip{p_{h},v}$ & $=\aip{f,v}$ & $\forall v\in V_{h}^{k}$\\[2mm]
$\aip{u_{h}, q} $ & $=0$ & $\forall q \in \f H_{h}^{k}$.
\end{tabular}
\end{equation}
We abbreviate by setting $\f X^{k}_{h} := V_{h}^{k-1} \times V_{h}^{k} \times \f H_{h}^{k}$. We must also assume the existence of bounded, surjective, and idempotent (projection) morphisms $\pi_{h}: V\to V_{h}$. It is generally not the orthogonal projection, as that fails to commute with the differentials. As a projection, it gives the following \keyterm{quasi-optimality} result:
\[
\vn[V]{u - \pi_{h} u} = \inf_{v \in V_{h}} \vn[V]{(I-\pi_{h})(u-v)} \leq \mn{I-\pi_{h}} \inf_{v\in V_{h}}\vn[V]{u-v}.
\]
The problem \eqref{eqn:mixed-hodge-laplacian-discrete} is then well-posed, with a Poincar\'e constant given by $c_{P}\mn{\pi_{h}^{k}}$, where $c_{P}$ is the Poincar\'e constant for the continuous problem. This guarantees all the previous abstract results apply to this case. With this, we have the following error estimate:%
\begin{theorem}[Arnold, Falk, and Winther~\cite{AFW2010}, Theorem 3.9]\label{thm:afw-main-mixed-errest}\index{error estimates!for the elliptic problem} Let $(V_{h},d)$ be a family of subcomplexes of the domain  $(V,d)$ of a closed Hilbert complex, parametrized by $h$ and admitting uniformly $V$-bounded cochain projections $\pi_{h}$, and let $(\sigma,u,p) \in \f X^{k}$ be the solution of the continuous problem and $(\sigma_{h},u_{h},p_{h}) \in \f X_{h}^{k}$ be the corresponding discrete solution. Then the following error estimate holds:
\begin{multline}\label{eqn:afw-main-mixed-error-est}
\vn[\f X]{(\sigma - \sigma_{h}, u - u_{h},p-p_{h})} = \vn[V]{\sigma-\sigma_{h}} + \vn[V]{u-u_{h}} + \vn{p-p_{h}} \\
\leq C( \inf_{\tau \in V_{h}^{k-1}} \vn[V]{\sigma-\tau} + \inf_{v \in V_{h}^{k}} \vn[V]{u-v}+ \inf_{q \in V_{h}^{k}} \vn[V]{p-q} + \mu \inf_{v \in V_{h}^{k}} \vn[V]{P_{\f B}u-v})
\end{multline}
with $\mu = \mu_{h}^{k} = \sup_{\substack{r \in \f H^{k}\\ \|r\| = 1}} \left\|\left(I-\pi_{h}^{k}\right)r\right\|$, the operator norm of $I-\pi_{h}^{k}$ restricted to $\f H^{k}$.
\end{theorem}%
\begin{corollary}
If the $V_{h}$ approximate $V$, that is, for all $u \in V$, $\inf_{v\in V_{h}} \vn[V]{u-v} \to 0$ as $h \to 0$, we have convergence of the approximations.
\end{corollary}
In general, the harmonic spaces $\f H^{k}$ and $\f H^{k}_{h}$ do not coincide, but they are isomorphic under many circumstances we shall consider (namely, the spaces are isomorphic if for all harmonic forms $q \in \f H^{k}$, the error $\vn{q - \pi_{h}q}$ is at most the norm $\vn{q}$ itself \cite[Theorem 3.4]{AFW2010}, and it \emph{always} holds for the de Rham complex). For a quantitative estimate relating the two different kinds of harmonic forms, we have the following
\begin{theorem}[\cite{AFW2010},Theorem 3.5]\label{thm:harmonic-approx} Let $(V,d)$ be a bounded, closed Hilbert complex, $(V_{h},d)$ a Hilbert subcomplex, and $\pi_{h}$ a bounded cochain projection. Then
\begin{align}
\vn[V]{(I-P_{\f H_{h}})q} &\leq \vn[V]{(I-\pi_{h}^{k})q}, \forall q \in \f H^{k}\\
\vn[V]{(I-P_{\f H})q} &\leq \vn[V]{(I-\pi_{h}^{k})P_{\f H}q}, \forall q \in \f H_{h}^{k}.
\end{align}
\end{theorem}

\subsection{Removing the Subcomplex Assumption: Variational Crimes}
For geometric problems, it is essential to remove the requirement that the approximating complex $V_{h}$ actually be subspaces of $V$. This is motivated by the example of approximating planar domains with curved boundaries by piecewise-linear approximations, resulting in finite element spaces that lie in a different function space \cite{Brae07}. Holst and Stern~\cite{HoSt10a} extend the Arnold, Falk, Winther~\cite{AFW2010} framework by supposing that $i_{h} : V_{h} \hookrightarrow V$ is an injective morphism which is not necessarily inclusion; they also require projection morphisms $\pi_{h} : V\to V_{h}$ with the property $\pi_{h} \circ i_{h} = \operatorname{id}$, which replaces the idempotency requirement of the preceding case. To summarize, given $(W,d)$ a Hilbert complex with domain $(V,d)$, $(W_{h},d_{h})$ another complex (whose inner product we denote $\aip[h]{\cdot,\cdot}$) with domain $(V_{h},d_{h})$, injective morphisms $i_{h} : W_{h}\hookrightarrow W$, and finally, projection morphisms $\pi_{h}: V \to V_{h}$. We then have the following generalized Galerkin problem:

\begin{equation}\label{eqn:mixed-hodge-laplacian-discrete-vcs}
\begin{tabular}{rll}
$\aip[h]{\sigma_{h}, \tau_{h}} - \aip[h]{u_{h}, d_{h}\tau_{h}}$ & $=0$ & $\forall \tau_{h} \in V_{h}^{k-1}$\\[2mm]
$\aip[h]{d_{h}\sigma_{h}, v_{h}} + \aip[h]{d_{h}u_{h},d_{h}v_{h}} + \aip[h]{p_{h},v_{h}}$ & $=\aip[h]{f_{h},v_{h}}$ & $\forall v_{h}\in V_{h}^{k}$\\[2mm]
$\aip[h]{u_{h}, q_{h}} $ & $=0$ & $\forall q_{h} \in \f H_{h}^{k}$,
\end{tabular}
\end{equation}
where $f_{h}$ is some interpolation of the given data $f$ into the space $W_{h}$ (we will discuss various choices of this operator later). This gives us a  bilinear form
\begin{multline}\label{eqn:hs-64-bilin-form}
B_{h}(\sigma_{h},u_{h},p_{h};\tau_{h},v_{h},q_{h}) :=\aip[h]{\sigma_{h}, \tau_{h}} - \aip[h]{u_{h}, d_{h}\tau_{h}} \\
+\aip[h]{d_{h}\sigma_{h}, v_{h}} + \aip[h]{d_{h}u_{h},d_{h}v_{h}} + \aip[h]{p_{h},v_{h}} 
-\aip[h]{u_{h},q_{h}}.
\end{multline}
This problem is well-posed, which again follows from the abstract theory as long as the complex is closed, and there is a corresponding Poincar\'e inequality:
\begin{theorem}[Holst and Stern~\cite{HoSt10a}, Theorem 3.5 and Corollary 3.6]\label{thm:disc-vcs-wellposed-poincare} Let $(V,d)$ and $(V_{h},d_{h})$ be bounded closed Hilbert complexes, with morphisms $i_{h} : V_{h} \hookrightarrow V$ and $\pi_{h} : V \to V_{h}$ such that $\pi_{h} \circ i_{h} = \operatorname{id}$. Then for all $v_{h} \in \f Z_{h}^{k\perp}$, we have 
\[
\|v_{h}\|_{V_{h}} \leq c_{P}\bvn{\pi_{h}^{k}}\bvn{i_{h}^{k+1}}\vn[V_{h}]{d_{h} v_{h}},
\]
where $c_{P}$ is the Poincar\'e constant corresponding to the continuous problem. If $(V,d)$ and $(V_{h},d_{h})$ are the domain complexes of closed complexes $(W,d)$ and $(W_{h},d_{h})$, then $\vn[V_{h}]{d_{h} v_{h}}$ is simply $\vn[h]{d_{h} v_{h}}$ (since it is the graph norm and $d^{2} = 0$).
\end{theorem}
In other words, the norm of the injective morphisms $i_{h}$ also contributes to the stability constant for this discrete problem. Analysis of this method results in two additional error terms (along with now having to explicitly reference the injective morphisms $i_{h}$ which may no longer be inclusions), due to the inner products in the space $V_{h}$ no longer necessarily being the restriction of that in $V$: the need to approximate the data $f$, and the failure of the morphisms $i_{h}$ to be unitary. 
\begin{theorem}[Holst and Stern~\cite{HoSt10a}, Corollary 3.11]\label{thm:main-hs-64-errest}\index{error estimates!for the elliptic problem!for variational crimes} Let $(V,d)$ be the domain complex of a closed Hilbert complex $(W,d)$, and $(V_{h},d_{h})$ the domain complex of $(W_{h},d_{h})$ with morphisms $i_{h} : W_{h}\to W$ and $\pi_{h} : V \to V_{h}$ as above. Then if we have solutions $(\sigma,u,p)$ and $(\sigma_{h},u_{h},p_{h})$ to \eqref{eqn:mixed-hodge-laplacian} and \eqref{eqn:mixed-hodge-laplacian-discrete-vcs} respectively, the following error estimate holds:
\begin{multline}\label{eqn:main-hs-64-errest}
\vn[V]{\sigma-i_{h}\sigma_{h}} + \vn[V]{u-i_{h}u_{h}} + \vn{p-i_{h}p_{h}} \\
\leq C( \inf_{\tau \in i_{h}V_{h}^{k-1}} \vn[V]{\sigma-\tau} + \inf_{v \in i_{h}V_{h}^{k}} \vn[V]{u-v}+ \inf_{q \in i_{h}V_{h}^{k}} \vn[V]{p-q} + \mu \inf_{v \in i_{h}V_{h}^{k}} \vn[V]{P_{\f B}u-v}\\
+ \vn[h]{f_{h} - i_{h}^{*} f} + \mn{I-J_{h}}\,\vn{f}),
\end{multline}
where $J_{h} = i_{h}^{*} i_{h}$, and $\mu = \mu_{h}^{k} = \sup\limits_{\substack{r \in \f H^{k} \\ \vn{r} = 1}}\left\|\left(I - i_{h}^{k} \pi_{h}^{k} \right)r\right\|$.
\end{theorem}
The extra terms (in the third line of the inequality) are analogous the terms described in the Strang lemmas \cite[\S III.1]{Brae07}. The main idea of the proof of Theorem \ref{thm:main-hs-64-errest} (which we will recall in more detail below, because we will need to prove a generalization of it as part of our main results) is to form an intermediate complex by pulling the inner products in the complex $(W,d)$ back to $(W_{h},d_{h})$ back by $i_{h}$, construct a solution to the problem there, and compare that solution with the solution we want. This modified inner product does not coincide with the given one on $W_{h}$ precisely when $i_{h}$ is not unitary:
\[
\aip[i_{h}^{*} W]{v,w} = \aip[h]{i_{h} v, i_{h} w} = \aip[h]{i_{h}^{*} i_{h} v, w} = \aip[h]{J_{h}v,w}.
\]
Unitarity is simply the condition $J_{h} = I$. The complex $W_{h}$ with the modified inner product now may be identified with a true subcomplex of $W$, for which the theory of \cite{AFW2010} directly applies, yielding a solution $(\sigma'_{h},u'_{h},p'_{h}) \in V_{h}^{k-1}\times V_{h}^{k} \times \f{H}_{h}^{\prime k}$, where $\f{H}_{h}^{\prime k}$ is the discrete harmonic space associated to the space with the modified inner product. This generally does not coincide with the discrete harmonic space $\f H^{k}_{h}$, since the discrete codifferential $d^{*\prime}_{h}$ in that case is defined to be the adjoint with respect to the modified inner product, yielding a different Hodge decomposition. The estimate of $\vn[V]{i_{h} \sigma_{h}'-\sigma} + \vn[V]{i_{h} u_{h}' - u} + \vn{i_{h}p'_{h}- p}$ then proceeds directly from the preceding theory for subcomplexes \eqref{eqn:afw-main-mixed-error-est}.  The variational crimes, on the other hand, arise from comparing the solution $(\sigma_{h},u_{h},p_{h})$ with $(\sigma_{h}',u_{h}',p_{h}')$. Finally, the error estimate \eqref{eqn:main-hs-64-errest} proceeds by the triangle inequality (and the boundedness of the morphisms $i_{h}$).

\subsection{Elliptic Error Estimates for a Nonzero Harmonic Part}\label{sec:our-extension}
Our objective in the remainder of this section is to prove one of our main results, a generalization of Theorem \ref{thm:main-hs-64-errest} which allows the possibility of the solution $u$ having a nonzero harmonic part $w$. We first need a few lemmas.

\begin{lemma}\label{thm:afw-main-mixed-errest-extended} Theorem \ref{thm:afw-main-mixed-errest} continues to apply when we have $\aip{u,q} = \aip{w,q}$ for all $q \in \f H^{k}$, where $w \in \f H^{k}$ is prescribed (i.e., $P_{\f H} u = w$, which may generally not be zero).
\end{lemma}
\begin{proof}

We closely follow the proof, in \cite{AFW2010}, of Theorem \ref{thm:afw-main-mixed-errest} above, noting where the modifications must occur. Let $B$ be the bounded bilinear form \eqref{eqn:afw-bilin-form}; then $(\sigma,u,p)$ satisfies, for all $(\tau_{h},v_{h},q_{h}) \in \f X^{k}_{h}$,
\[
B(\sigma,u,p;\tau_{h},v_{h},q_{h}) = \aip{f,v_{h}} - \aip{u,q_{h}}.
\]
We $V$-orthogonally project $(\sigma, u, p)$ in each factor to $(\tau, v, q) \in \f X^{ k}_{h}$. Then for any $(\tau_{h},v_{h},q_{h}) \in \f X^{k}_{h}$,
\begin{multline}\label{eqn:afw-bilin-derivation}	
B(\sigma_{h}-\tau,u_{h}-v,p_{h}-q; \tau_{h},v_{h},q_{h})\\
= B(\sigma-\tau,u-v,p-q;\tau_{h},v_{h},q_{h}) + \aip{u,q_{h}} - \aip{w, q_{h}}\\
= B(\sigma-\tau,u-v,p-q;\tau_{h},v_{h},q_{h}) + \aip{P_{\f H_{h}}(u -w),q_{h}}\\
\leq C\left(\vn[V]{\sigma-\tau} + \vn[V]{u-v} +\vn{p-q} + \bvn{P_{\f H_{h}}(u - w)}\right)(\vn[V]{\tau_{h}} + \vn[V]{v_{h}} + \vn{q_{h}}).
\end{multline}
Noticing that the factor $p_{h}-q$ in the bilinear form above is in the original domain $\f H^{ k}_{h}$, we can now choose the appropriate $(\tau_{h},v_{h},q_{h})$ that verifies $\inf$-$\sup$ condition of $B$:
\begin{multline*}
B(\sigma_{h}-\tau,u_{h}-v,p_{h}-q; \tau_{h},v_{h},q_{h}) \\
\geq \gamma (\vn[V]{\sigma_{h}-\tau} + \vn[V]{u_{h}-v} + \vn{p_{h}-q})(\vn[V]{\tau_{h}} + \vn[V]{v_{h}} + \vn{q_{h}}).
\end{multline*}
Comparing this to \eqref{eqn:afw-bilin-derivation} above, we may cancel the common factor, and divide by $\gamma$ to arrive at
\begin{multline}
\vn[V]{\sigma_{h}-\tau} + \vn[V]{u_{h}-v} + \vn{p_{h}-q}\\ \leq C\gamma^{-1}\left(\vn[V]{\sigma-\tau} + \vn[V]{u-v} +\vn{p-q} + \bvn{P_{\f H_{h}}(u - w)}\right).
\end{multline}
This differs (aside from the notation) from \cite{AFW2010} in that we now have, rather than $P_{\f H_{h}} u$, instead $P_{\f H_{h}} (u -w)$, with the harmonic part subtracted off. Removing the harmonic part allows us to continue as in \cite{AFW2010}: the Hodge decomposition $u-w=u -  P_{\f H} u$ consists only of coboundary and perpendicular terms $u_{\f B} + u_{\perp} \in \f B^{k} \oplus \f Z^{k\perp}$. With $\f H^{ k}_{h}$ contained in $\f Z^{k}$, it follows $P_{\f H_{h}} u_{\perp} = 0$, and $P_{\f H_{h}}\pi_{h} u_{\f B} = 0$. Also, $(I-\pi_{h})u_{\f B}$ is perpendicular to $\f H^{k}$. Therefore, for all $q \in \f H^{k}_{h}$, 
\begin{multline*}
\aip{P_{\f H_{h}} (u - P_{\f H} u),q} = \aip{P_{\f H_{h}} u_{\f B},q}= \aip{P_{\f H_{h}} (u_{\f B} - \pi_{h} u_{\f B}),q}\\ =\aip{u_{\f B} - \pi_{h} u_{\f B} ,q} =\aip{u_{\f B} - \pi_{h} u_{\f B} ,(I- P_{\f H})q} .
\end{multline*}
Now, setting \[
q = \frac{P_{\f H_{h}}(u - P_{\f H} u)}{\vn{P_{\f H_{h}}(u - P_{\f H} u)}} \in \f H^{ k}_{h},
\]
we have
\begin{multline*}
 \vn{P_{\f H_{h}}(u - P_{\f H} u)}= \aip{P_{\f H_{h}} (u - P_{\f H} u),q} =\aip{u_{\f B} - \pi_{h} u_{\f B},(I- P_{\f H})q}\\ \leq \vn{u_{\f B} - \pi_{h} u_{\f B}}\, \vn{(I- P_{\f H}) q} \leq C\vn{(I-P_{\f H})q} \inf_{v \in V^{k}_{h}} \vn[V]{u_{\f B} - v}. 
\end{multline*}
Finally, by the second estimate of Theorem \ref{thm:harmonic-approx} above, we can bound $\vn{(I - P_{\f H})q}$ by $\vn{(I -\pi_{h}) P_{\f H} q}$, giving us
\[
\vn{(I-P_{\f H})q} \leq \vn{(I -\pi_{h}) P_{\f H} q} \leq \sup_{\substack{\vn{r}=1\\ r\in \f H^{k}}}\vn{(I - \pi_{h})r}\,\vn{P_{\f H}q} \leq \mu.
\]
From the triangle inequality, we derive the estimate 
\begin{multline*}
\vn[V]{\sigma - \sigma_{h}} + \vn[V]{u-u_{h}} +\vn{p-p_{h}} \\ \leq \vn[V]{\sigma - \tau}  + \vn[V]{u-v}+\vn{p-q} +\vn[V]{\tau-\sigma_{h}} + \vn[V]{u_{h}-v}   + \vn{q-p_{h}} \\
\leq (1+ C\gamma^{-1})\left( \vn[V]{\sigma - \tau}  + \vn[V]{u-v}+\vn{p-q} + \mu \inf_{v \in i_{h} V^{h}_{k}} \vn[V]{P_{\f B} u - v}\right).
\end{multline*}
Using best approximation property of orthogonal projections, we can express the remaining terms with the infima, and this gives the result.
\end{proof}
We also need a technical lemma which enables us to identify the orthogonal projection onto the identified subcomplex $i_{h} \f X^{\prime k}_{h}$ in order to be able to make additional estimates of the variational crimes in terms of the operator norms $\mn{I-J_{h}}$. It is the infinite-dimensional analogue of taking a \index{Moore-Penrose pseudoinverse}Moore-Penrose pseudoinverse~\cite[\S 3.3]{Strang88} for infinite-dimensional spaces:
\begin{lemma} Let $i_{h} : W_{h} \to W$ be an injective map of Hilbert spaces, and $J = i_{h}^{*}i_{h}$. Then $J_{h}$ is invertible, and $J_{h}^{-1}i_{h}^{*}$ is the Moore-Penrose pseudoinverse of $i_{h}$, i.e. it maps $i_{h} W_{h}$ isometrically back to $W_{h}$ with the modified inner product.
\end{lemma}
We write $i_{h}^{+}$ for $J_{h}^{-1} i_{h}^{*}$.
\begin{proof}
The invertibility of $J_{h}$ follows directly from the injectivity of $i_{h}$, which makes $\aip[h]{J_{h}\cdot,\cdot}$ a positive-definite form. Now, $(J_{h}^{-1} i^{*}_{h})i_{h} = J^{-1}_{h}J_{h} = \operatorname{id}_{W_{h}}$, which shows that it is in fact a left inverse, as required for pseudoinverses. To show the orthogonality, minimizing $\frac{1}{2} \vn{i_{h}w - b}^{2}$ for any $b \in W$ yields, by the completeness of $W_{h}$, the solution $w=J_{h}^{-1}i_{h}^{*}b$, showing that it is a least squares solution, therefore the Moore-Penrose pseudoinverse.
\end{proof}
We are now ready to prove our main elliptic error estimate, an extension of Theorem \ref{thm:main-hs-64-errest}. 
\begin{theorem}[Extension of elliptic error estimates to allow for a harmonic part]\label{thm:main-estimate-nonzero-harmonic}\index{error estimates!for the elliptic problem!extension to handle nonzero harmonic} Consider the problems \eqref{eqn:mixed-hodge-laplacian} and \eqref{eqn:mixed-hodge-laplacian-discrete-vcs} but instead with now with prescribed, possibly nonzero harmonic part $w$: Given $f \in W^{k}$ and $w \in \f H^{k}$, we seek $(\sigma,u,p) \in \f X^{k}$ such that
\begin{equation}\label{eqn:mixed-hodge-laplacian-nonzero-harm}
\begin{tabular}{rll}
$\aip{\sigma, \tau} - \aip{u, d\tau}$ & $=0$ & $\forall \tau \in V^{k-1}$\\
$\aip{d\sigma, v} + \aip{du,dv} + \aip{p,v}$ & $=\aip{f,v}$ & $\forall v\in V^{k}$\\
$\aip{u, q} $ & $=\aip{w,q}$ & $\forall q \in \f H^{k}$.
\end{tabular}
\end{equation}
The solution to this problem exists and is unique, with $w$ indeed equal to $P_{\f H} u$, and is bounded by $c(\vn{f} + \vn{w})$, with $c$ depending only on the Poincar\'e constant. Now, consider the discrete problem, with $f_{h}, w_{h} \in V^{k}_{h}$:
\begin{equation}\label{eqn:mixed-hodge-laplacian-discrete-vcs-nonzero-harm}
\begin{tabular}{rll}
$\aip[h]{\sigma_{h}, \tau_{h}} - \aip[h]{u_{h}, d_{h}\tau_{h}}$ & $=0$ & $\forall \tau_{h} \in V_{h}^{k-1}$\\[2mm]
$\aip[h]{d_{h}\sigma_{h}, v_{h}} + \aip[h]{d_{h}u_{h},d_{h}v_{h}} + \aip[h]{p_{h},v_{h}}$ & $=\aip[h]{f_{h},v_{h}}$ & $\forall v_{h}\in V_{h}^{k}$\\[2mm]
$\aip[h]{u_{h}, q_{h}} $ & $=\aip[h]{w_{h},q_{h}}$ & $\forall q_{h} \in \f H_{h}^{k}$.
\end{tabular}
\end{equation}
This problem is also well-posed, with the modified Poincar\'e constant in Theorem \ref{thm:disc-vcs-wellposed-poincare}. Then we have the following generalization of the error estimate \eqref{eqn:main-hs-64-errest} above:
\begin{multline}\label{eqn:main-elliptic-errest}
\vn[V]{\sigma-i_{h}\sigma_{h}} + \vn[V]{u-i_{h}u_{h}} + \vn{p-i_{h}p_{h}} \\
\leq C\left( \inf_{\tau \in i_{h}V_{h}^{k-1}} \vn[V]{\sigma-\tau} + \inf_{v \in i_{h}V_{h}^{k}} \vn[V]{u-v}+ \inf_{q \in i_{h}V_{h}^{k}} \vn[V]{p-q} + \mu \inf_{v \in i_{h}V_{h}^{k}} \vn[V]{P_{\f B}u-v}\right.\\
\left.\vphantom{\inf_{\tau \in i_{h}V_{h}^{k-1}}}+\inf_{\xi \in i_{h} V^{k}_{h}}\vn[V]{w-\xi}+ \vn[h]{f_{h} - i_{h}^{*} f} +  \vn[h]{w_{h} - i_{h}^{*} w}+\mn{I-J_{h}}\,(\vn{f}+\vn{w})\right),
\end{multline}
where, as before, $J_{h} = i_{h}^{*} i_{h}$, and $\mu = \mu_{h}^{k} = \sup\limits_{\substack{r \in \f H^{k} \\ \vn{r} = 1}}\left\|\left(I - i_{h}^{k} \pi_{h}^{k} \right)r\right\|$.
\end{theorem}
We see that three new error terms arise from the approximation of the harmonic part, one being the data interpolation error (but measured in the $V_{h}$-norm, which partially captures how $d$ fails to commute with $i_{h}^{*}$ and how $w_{h}$ may not necessarily be a discrete harmonic form), another best approximation term, and finally another term from the non-unitarity. The relation of $f_{h}$ to $f$ and $w_{h}$ to $w$ need not be further specified, because the theorem directly expresses such a dependence in terms of their relation to $i_{h}^{*} f$ and $i_{h}^{*}w$; it has been isolated as a separate issue. However as mentioned in the introduction, and following~\cite{HoSt10a}, we often take $f_{h} = \Pi_{h}f$, where $\Pi_{h}$ is some family of linear interpolation operators with $\Pi_{h} \circ i_{h} = \operatorname{id}$. Another seemingly obvious choice is $i_{h}^{*}$ itself (thus making those corresponding error terms zero), but as mentioned in~\cite{HoSt10a}, this can be difficult to compute, so we do not restrict ourselves to this case. Various choices of interpolation will be crucial in deciding which estimates to make in the parabolic problem. We split the proof of this theorem into two parts, the first of which derives the quantities on the second line of \eqref{eqn:main-elliptic-errest}, and the second part, we derive the quantities on the third line of \eqref{eqn:main-elliptic-errest}. Generally, we follow the pattern of proof in \cite[Theorem 3.9]{AFW2010} and \cite[Theorem 3.10]{HoSt10a}, noting the necessary modifications, as well as a similar technique given for the improved error estimates by Arnold and Chen~\cite{AC2012}.
\begin{proof}[First part of the proof of Theorem \ref{thm:main-estimate-nonzero-harmonic}]
First, following Holst and Stern~\cite{HoSt10a} as above, we construct the complex $W_{h}$ but with the modified inner product $\aip{J_{h}\cdot,\cdot}$ (the associated domain complex $V_{h}$ remains the same). This gives us a discrete Hodge decomposition with another type of orthogonality and corresponding discrete harmonic forms and orthogonal complement (due to a different adjoint $d^{*\prime}_{h}$):
\[
V_{h}^{k} = \f B_{h}^{k} \oplus \f H^{\prime k}_{h} \oplus \f Z^{k\perp\prime}_{h}
\]
(generally, primed objects will represent the corresponding objects defined with the modified inner product; the discrete coboundaries are in fact the same as before, because $d$ and $d_{h}$ do not depend on the choice of inner product). The main complications arise in having to keeping careful track of the different harmonic forms involved, because their nonequivalence and possible non-preservation by the operators contribute directly to the error. We then define, as in \cite{HoSt10a}, the intermediate solution $(\sigma'_{h},u'_{h},p'_{h}) \in V_{h}^{k-1}\times V_{h}^{k} \times \f H_{h}^{\prime k}$ (which we abbreviate as $ \f X^{\prime k}_{h}$):
\begin{equation}\label{eqn:mixed-hodge-laplacian-discrete-intermediate-nonzero-harm}
\begin{tabular}{rll}
$\aip[h]{J_{h}\sigma'_{h}, \tau_{h}} - \aip[h]{J_{h}u'_{h}, d_{h}\tau_{h}}$ & $=0$ & $\forall \tau_{h} \in V_{h}^{k-1}$\\[2mm]
$\aip[h]{J_{h}d_{h}\sigma'_{h}, v_{h}} + \aip[h]{J_{h}d_{h}u'_{h},d_{h}v_{h}} + \aip[h]{J_{h}p'_{h},v_{h}}$ & $=\aip[h]{i_{h}^{*}f,v_{h}}$ & $\forall v_{h}\in V_{h}^{k}$\\[2mm]
$\aip[h]{J_{h}u'_{h}, q'_{h}} $ & $=\aip[h]{i^{*}_{h}w,q'_{h}}$ & $\forall q'_{h} \in \f H_{h}^{\prime k}$,
\end{tabular}
\end{equation}
and the corresponding bilinear form $B'_{h} : \f X'_{h} \times \f X'_{h} \to \R$ given by
\begin{multline}\label{eqn:hs-64-intermediate-bilin-form}
B'_{h}(\sigma'_{h},u'_{h},p'_{h};\tau_{h},v_{h},q'_{h}) :=\aip[h]{J_{h}\sigma'_{h}, \tau_{h}} - \aip[h]{J_{h}u'_{h}, d_{h}\tau_{h}} \\
+\aip[h]{J_{h}d_{h}\sigma'_{h}, v_{h}} + \aip[h]{J_{h}d_{h}u'_{h},d_{h}v_{h}} + \aip[h]{J_{h}p'_{h},v_{h}} 
-\aip[h]{J_{h}u'_{h},q'_{h}}.
\end{multline}
This satisfies the inf-sup condition with Poincar\'e constant $c_{P}\mn{\pi_{h}}$. Note that we will need to extend all the bilinear forms $B_{h}$, and $B'_{h}$ in the last factor to all of $V_{h}^{k}$ in order to compare the two, since they are initially only defined on the respective, \emph{differing} harmonic form spaces. This is not a problem so long as we remember to invoke the inf-sup condition only when using the non-extended versions. The idea is, again, to use the triangle inequality:
\begin{align}
\vn[V]{\sigma-i_{h}\sigma_{h}} &+ \vn[V]{\tau - i_{h}\tau_{h}} + \vn{p - i_{h}p_{h}} \leq \\
&  \vn[V]{\sigma-i_{h}\sigma'_{h}} + \vn[V]{\tau - i_{h}\tau'_{h}} + \vn{p - i_{h}p'_{h}} \label{eqn:afw-intermediate-est} \\
+&\vn[V]{i_{h}(\sigma'_{h}-\sigma_{h})} + \vn[V]{i_{h}(\tau'_{h} - \tau_{h})} + \vn{i_{h}(p'_{h} - p_{h})}\label{eqn:hs-vc-est}.
\end{align}
These quantities can be estimated using only geometric properties of the domain; we have no need to actually explicitly compute $(\sigma'_{h},u'_{h},p'_{h})$. To estimate the term \eqref{eqn:afw-intermediate-est} (which we shall refer to as the \textsc{pde} approximation term, whereas \eqref{eqn:hs-vc-est} will be called variational crimes), we recall that $i_{h}$ is an isometry of $W_{h}$ with the modified inner product to its image, which is a subcomplex.

Thus, Lemma \ref{thm:afw-main-mixed-errest-extended} above applies, with the approximation $(i_{h}\sigma'_{h},i_{h}u'_{h},i_{h}p'_{h})$ on identified subcomplex $i_{h} \f X^{\prime k}_{h}$. This gives us the terms on the second line of \eqref{eqn:main-elliptic-errest}.
\end{proof}
To finish our main proof, we need to consider the variational crimes \eqref{eqn:hs-vc-est}. Since the maps $i_{h}$ are bounded, and we eventually absorb their norms into the constant $C$ above, it suffices to consider $\vn[V_{h}]{\sigma_{h}-\sigma'_{h}} + \vn[V_{h}]{u_{h}-u'_{h}} + \vn[h]{p_{h}-p'_{h}}$, which we shall state as a separate theorem.
\begin{theorem}\label{thm:main-elliptic-estimates-part2} Let $(\sigma_{h},u_{h},p_{h}) \in \f X^{k}_{h}$ be a solution to \eqref{eqn:mixed-hodge-laplacian-discrete-vcs-nonzero-harm}, $(\sigma'_{h},u'_{h},p'_{h}) \in \f X^{\prime k}_{h}$ a solution to \eqref{eqn:mixed-hodge-laplacian-discrete-intermediate-nonzero-harm}, and $w =  P_{\f H} u$, the prescribed harmonic part of the continuous problem. Then
\begin{multline}
\vn[V_{h}]{\sigma_{h}-\sigma'_{h}} + \vn[V_{h}]{u_{h}-u'_{h}} + \vn[h]{p_{h}-p'_{h}} \\
\leq C(\vn[h]{f_{h}- i_{h}^{*} f} + \vn[V_{h}]{ w_{h}- i_{h}^{*} w}+ \mn{I-J_{h}}(\vn{f} + \vn{w}) + \inf_{\xi \in i_{h}V_{h}^{k}} \vn[V]{w - \xi}),
\end{multline}
i.e., they are bounded by the terms on the third line in \eqref{eqn:main-elliptic-errest}.
\end{theorem}
\begin{proof}[Proof of Theorem \ref{thm:main-elliptic-estimates-part2} and second part of the proof of Theorem \ref{thm:main-estimate-nonzero-harmonic}]
We follow the proof of Holst and Stern~\cite[Theorem 3.10]{HoSt10a} and note the modifications. Let $(\tau,v,q)$ and $(\tau_{h},v_{h},q_{h}) \in \f X_{h}^{k}$. Consider the bilinear form $B_{h}$, \eqref{eqn:hs-64-bilin-form} above, and write
\begin{multline*}
B_{h}(\sigma_{h}-\tau,u_{h}-v,p_{h}-q; \tau_{h},v_{h},w_{h}) = B_{h}(\sigma_{h}- \sigma'_{h}, u_{h}-u'_{h},p_{h}-p'_{h};\tau_{h},v_{h},q_{h})\\
+ B_{h}(\sigma'_{h}-\tau,u'_{h}-v,p'_{h}-q;\tau_{h},v_{h},q_{h}).
\end{multline*}
We then have, recalling the modified bilinear form $B'_{h}$, \eqref{eqn:hs-64-intermediate-bilin-form} above,  and extending it in the last factors to all of $V_{h}^{k}$,
\begin{multline*}
B_{h}(\sigma_{h}',u_{h}',p_{h}';\tau_{h},v_{h},q_{h}) \\= B_{h}'(\sigma_{h}',u_{h}',p_{h}';\tau_{h},v_{h},q_{h})
+ \aip[h]{(I-J_{h}) \sigma_{h}',\tau_{h}} - \aip[h]{(I-J_{h})u_{h}',d_{h}\tau_{h}}\\
+\aip[h]{(I-J_{h}) d_{h}\sigma_{h}',v_{h}} +\aip[h]{(I-J_{h}) d_{h}u_{h}',d_{h}v_{h}} +\aip[h]{(I-J_{h}) p_{h}',v_{h}}\\
-\aip[h]{(I-J_{h}) u_{h}',q_{h}}.
\end{multline*}
Substituting the respective solutions \eqref{eqn:mixed-hodge-laplacian-discrete-vcs-nonzero-harm} and \eqref{eqn:mixed-hodge-laplacian-discrete-intermediate-nonzero-harm} (and noting the slight discrepancy in the use of different harmonic forms), we have
\begin{align*}
B_{h}'(\sigma_{h}',u_{h}',p_{h}';\tau_{h},v_{h},q_{h}) &= \aip[h]{i_{h}^{*}f,v_{h}}-\aip[h]{J_{h}u_{h}',q_{h}}\\
B_{h}(\sigma_{h},u_{h},p_{h};\tau_{h},v_{h},q_{h}) &= \aip[h]{f_{h},v_{h}} - \aip[h]{w_{h},q_{h}},
\end{align*}
so
\begin{multline*}
B_{h}(\sigma_{h}-\sigma_{h}',u_{h}-u_{h}',p_{h}-p_{h}';\tau_{h},v_{h},q_{h}) \\
 = \aip[h]{f_{h} - i_{h}^{*} f,v_{h}} +\aip[h]{u_{h}', q_{h}} -\aip[h]{w_{h},q_{h}}\\
- \aip[h]{(I-J_{h}) \sigma_{h}',\tau_{h}} + \aip[h]{(I-J_{h})u_{h}',d_{h}\tau_{h}}\\
-\aip[h]{(I-J_{h}) d_{h}\sigma_{h}',v_{h}} -\aip[h]{(I-J_{h}) d_{h}u_{h}',d_{h}v_{h}} -\aip[h]{(I-J_{h}) p_{h}',v_{h}}.
\end{multline*}
As before, we bound the form above and below. For the upper bound, using Cauchy-Schwarz to estimate the extra inner product terms, we arrive at
\begin{multline*}
B_{h}(\sigma_{h}-\tau,u_{h}-v,p_{h}-q;\tau_{h},v_{h},q_{h})\\
\leq C\left(\vn[h]{f_{h} - i_{h}^{*} f} + \vn[h]{P_{\f H_{h}}( u'_{h} - w_{h})} + \mn{I-J_{h}} (\vn[V_{h}]{\sigma'_{h}} + \vn[V_{h}]{u'_{h}} + \vn[h]{p'_{h}})\right.\\
\left.+\vn[V_{h}]{\sigma_{h}' -\tau} + \vn[V_{h}]{u_{h}'-v} + \vn[h]{p_{h}'-q}\right)\left(\vn[V_{h}]{\tau_{h}} + \vn[V_{h}]{v_{h}} + \vn[h]{q_{h}}\right).
\end{multline*}

For the lower bound, we again choose $(\tau_{h},\sigma_{h},q_{h}) \in \f X^{k}_{h}$ to verify the $\inf$-$\sup$ condition this time for $B_{h}$:
\begin{multline*}
B_{h}(\sigma_{h}-\tau,u_{h}-v,p_{h}-q;\tau_{h},v_{h},q_{h})\\
\geq \gamma_{h} \left(\vn[V_{h}]{\sigma_{h}-\tau}+\vn[V_{h}]{u_{h}-v}+\vn[h]{p_{h}-q}\right)\left(\vn[V_{h}]{\tau_{h}} + \vn[V_{h}]{v_{h}} + \vn[h]{q_{h}}\right)
\end{multline*}
and $\gamma_{h}$ depends only on the Poincar\'e constant $c_{P} \mn{i_{h}}\,\mn{\pi_{h}}$, uniformly bounded in $h$. Comparing with the upper bound and dividing out the common factor as before,
this leads to:
\begin{multline*}
\vn[V_{h}]{\sigma_{h}-\tau}+\vn[V_{h}]{u_{h}-v}+\vn[h]{p_{h}-q}\\
\leq C\gamma_{h}^{-1}\left(\vn[h]{f_{h} - i_{h}^{*} f} + \vn[h]{P_{\f H_{h}}( u'_{h} - w_{h})} + \mn{I-J_{h}} (\vn[V_{h}]{\sigma'_{h}} + \vn[V_{h}]{u'_{h}} + \vn[h]{p'_{h}})\right.\\
\left.+\vn[V_{h}]{\sigma_{h}' -\tau} + \vn[V_{h}]{u_{h}'-v} + \vn[h]{p_{h}'-q}\right).
\end{multline*}
Choosing $(\tau,v,q)=(\sigma_{h}',u'_{h},P_{\f H_{h}}p'_{h})$, applying the triangle inequality with $p_{h}'$ to account for the mismatch in the harmonic spaces, and using the well-posedness of the continuous problem \eqref{eqn:mixed-hodge-laplacian-discrete-intermediate-nonzero-harm},
\begin{multline*}
\vn[V_{h}]{\sigma_{h}-\sigma_{h}'}+\vn[V_{h}]{u_{h}-u_{h}'} + \vn[h]{p_{h}-p'_{h}}\\
\leq C\left(\vn[h]{f_{h} - i_{h}^{*} f} + \vn[h]{P_{\f H_{h}}( u'_{h} - w_{h})} + \mn{I-J_{h}} (\vn{f}+\vn{w}) + \vn[h]{p_{h}'-q}\right).
\end{multline*}
This differs from \cite{HoSt10a} in that we have the bound in terms of $\vn{f} +\vn{w}$, and that we must estimate $\vn[h]{P_{\f H_{h}}( u'_{h} - w_{h})}$ rather than $\vn[h]{P_{\f H_{h}}u'_{h}}$ alone. First, we use the modified Hodge decomposition to uniquely write $u_{h}'$ as  $u'_{\f B} + P_{\f H'_{h}} u'_{h}+ u'_{\perp} $ with $u'_{\f B} \in \f B^{k}_{h}$ and $u'_{\perp} \in \f Z^{k\perp\prime}_{h}$, and
\[
\vn[h]{P_{\f H_{h}}( u'_{h} - w_{h})}  \leq \vn[h]{P_{\f H_{h}} (u'_{\f B}+u'_{\perp})} + \vn[h]{P_{\f H_{h}}(P_{\f H'_{h}} u'_{h} - w_{h})}.
\]
(The projection $P_{\f H'_{h}}$ is respect to the \emph{modified} inner product). For the first term, we proceed exactly as in \cite{HoSt10a}: we have $P_{\f H_{h}} u'_{\f B} = 0$ since the coboundary space is still the same, and thus only the term $u'_{\perp}$  contributes. Now $u'_{\perp} \in \f Z_{h}^{k\perp\prime}$ so, using $J_{h}$ to express it in terms of $V$-orthogonality, we have $J_{h} u'_{\perp} \perp \f Z^{k}_{h}$, and thus $P_{\f H_{h}}J_{h} u'_{\perp}=0$. Therefore, we have 
\[
\vn[h]{P_{\f H_{h}} (u'_{\f B}+u'_{\perp})}  = \vn[h]{P_{\f H_{h}} u'_{\perp}}=\vn[h]{P_{\f H_{h}} (I-J_{h})u'_{\perp}} \leq C\mn{I-J_{h}}(\vn{f}+\vn{w}).
\]
For the $p_{h}'$ term, this also proceeds as in \cite{HoSt10a} unchanged (except for, of course, the extra $\vn{w}$ term): using the (unmodified) discrete Hodge decomposition, we have $p_{h}' = P_{\f B_{h}} p_{h}' + P_{\f H_{h}}p_{h}' = P_{\f B_{h}}p_{h}' + q$. Since $p_{h}' \in \f H_{h}^{\prime k}$, a similar argument gives $J_{h}p'_{h} \perp \f B_{h}^{k}$, so $P_{\f B_{h}} J_{h}p'_{h} = 0$ and
\[
\vn[h]{p'_{h}-q} = \vn[h]{P_{\f B_{h}} p_{h}'} = \vn[h]{P_{\f B_{h}} (I-J_{h})p_{h}'}\leq C\mn{I-J_{h}}(\vn{f}+\vn{w}). 
\]
Finally, we must consider the term $\vn[h]{P_{\f H} (P_{\f H'_{h}} u'_{h} - w_{h})}$. Expressing $u_{h}'$ in terms of $w$, the terms do not combine as easily as the analogous terms involving $f_{h}$ and $i_{h}^{*}f$, because their action as linear functionals operate on different harmonic spaces. 

Continuing with the proof of the theorem, we recall the third equation of \eqref{eqn:mixed-hodge-laplacian-discrete-intermediate-nonzero-harm}:
 \[
\aip[h]{J_{h} u'_{h},q'} = \aip[h]{i_{h}^{*} w,q'} = \aip[h]{J_{h}(J_{h}^{-1} i_{h}^{*} w),q'}
\]
which therefore says $P_{\f H_{h}'} u_{h}' = P_{\f H_{h}'} i_{h}^{+}w$. This enables us to properly work with the modified orthogonal projection $P_{\f H'_{h}}$. Because $i_{h}^{+}$ is an isometry of the subspace $i_{h} W_{h}$ to $W_{h}$ , we have
\[
P_{\f H'_{h}}i_{h}^{+} w =  i_{h}^{+} P_{i_{h}\f H'_{h}} w.
\]
where now $P_{i_{h}\f H'_{h}}$ is the orthogonal projection onto the identified image harmonic space. Then, using the triangle inequality again,
\begin{multline*}
\vn[h]{P_{\f H_{h}}(P_{\f H_{h}'} u'_{h} - w_{h})}\\ \leq \bvn[h]{P_{\f H_{h}}\left(P_{\f H'_{h}}i_{h}^{+} w  - i_{h}^{+} w\right)} + \vn[h]{P_{\f H_{h}}(J_{h}^{-1} i_{h}^{*} w - i_{h}^{*}w)}+ \vn[h]{P_{\f H_{h}}(i_{h}^{*}w - w_{h})}\\
\leq \mn{P_{\f H_{h}}}\left( \mn{i_{h}^{+}}\bvn{\left(I-P_{i_{h}\f H'_{h}}\right) w} + \mn{J_{h}^{-1}} \;\mn{I-J_{h}}\, \vn[h]{i^{*}_{h}w} + \vn[h]{i^{*}_{h} w - w_{h}}\right)\\
\leq C\left(\bvn {\left(I- P_{i_{h} \f H'_{h}}\right)w} + \mn{I-J_{h}} \;\vn{w} + \vn[h]{i_{h}^{*} w - w_{h}}\right).
\end{multline*}
The last term is the data approximation error for $w$, and the second term combines with the previous errors that reflect the non-unitarity of the operator. So, all that remains is to estimate the first term. Since it is in the subcomplex $i_{h} W_{h}$, the first estimate of Theorem \ref{thm:harmonic-approx} applies:
\begin{equation}\label{eqn:I-PH-to-inf}
\bvn{\left(I-P_{i_{h}\f H'_{h}}\right)w}\leq \vn{(I-\pi_h')w}\leq C \inf_{\xi\in i_h V^{k}_h} \vn[V]{w-\xi},
\end{equation}
by quasi-optimality.
\end{proof}

\begin{proof}[Concluding remarks of the proof of Theorem \ref{thm:main-estimate-nonzero-harmonic}]To summarize, we proved our Main Theorem \ref{thm:main-estimate-nonzero-harmonic} by defining an intermediate solution on a modified complex that we identify with a subcomplex, and analyzing the result via the Arnold, Falk, and Winther~\cite{AFW2010} framework. That theorem holds, with the estimate unchanged, though now $u$ and $u_{h}$ no longer are perpendicular to their respective harmonic spaces. The place where the extra terms all show up is in the variational crimes. In the process of arriving at a term that looks like $i_{h}^{*}w - w_{h}$, working with the different harmonic forms produces two more non-unitarity terms $\mn{I-J_{h}} (\vn{f} + \vn{w})$, and finally, using Theorem \ref{thm:harmonic-approx} yields a direct estimate of how $w$ fails to be a modified discrete harmonic form, giving the last best approximation term $\inf_{\xi\in i_h V^{k}_h} \vn[V]{w-\xi}$.
\end{proof}

We also note for future reference that in spaces where we have improved error estimates (which means $\pi_{h}$ are $W$-bounded maps) that we can replace that last $V$-norm in \eqref{eqn:I-PH-to-inf} to be the $W$-inner product. Finally, we remark that, for a certain types of data interpolation, the errors $\vn{f_{h} - i_{h}^{*} f}$ and $\vn{w_{h} - i_{h}^{*} w}$ can be rewritten in terms of the other errors and another best approximation term. This will be useful for us in our examples.
\begin{theorem}[Holst and Stern \cite{HoSt10a}, Theorem 3.12]\label{thm:fam-of-projections}
If $\Pi_{h}:W^{k} \to W_{h}^{k}$ is a family of linear projections uniformly bounded with respect to $h$, then for all $f \in W^{k}$,
\begin{equation}
\vn{\Pi_{h} f- i_{h}^{*} f}\leq C\left(\mn{I-J_{h}} \, \vn{f } + \inf_{\phi \in i_{h}W_{h}^{k}}\vn{f-\phi}\right).
\end{equation}

\end{theorem}

\section{Abstract Evolution Problems}
\label{sec:abs-evol}
In order to solve and approximate linear evolution problems, we introduce the framework of Bochner spaces (also following Gillette, Holst, and Zhu \cite{GiHo11a}), which realizes time-dependent functions as curves in Banach spaces (which will correspond to spaces of spatially-dependent functions in our problem). We follow mostly \cite{RR2004} and \cite{Ev98} for this material.
\subsection{Overview of Bochner Spaces and Abstract Evolution Problems}
Let $X$ be a Banach space and $I:=[0,T]$ an interval in $\R$ with $T>0$. We define
\[C(I,X) := \{u: I\raw X\;\;|\;\;\text{$u$ bounded and continuous}\}.\]
In analogy to spaces of continuous, real-valued functions, we define a supremum norm on $C(I,X)$,
making $C(I, X)$ into a Banach space:
\[\vn{u}_{C( I,X)} := \sup_{t\in  I}\vn{u(t)}_X.\]

We will of course need to deal with norms other than the supremum norm, which motivates us to define \keyterm{Bochner spaces}: to define $\e L^p( I,X)$, we complete $C(I,X)$ with the norm
\[\vn{u}_{L^p( I,X)} := \left(\int_{ I}\vn{u(t)}^p_X dt\right)^{1/p}.\]
Similarly, we have the space $H^1(I,X)$, the completion of $C^{1}(I,X)$ with the norm
\[\vn{u}_{H^1( I,X)} := \left(\int_{ I}\vn{u(t)}^2_X + \bvn{\frac d{dt}u(t)}^2_X dt\right)^{1/2}.\]
There are methods of formulating this in a more measure-theoretic way (\cite[Appendix E]{Ev98}), considering Lebesgue-measurable subsets of $I$.

As mentioned before, for our purposes, $X$ will be some space of spatially-dependent functions, and the time-dependence is captured as being a curve in this function space (although this interpretation is only correct when we are considering $C(I,X)$---we must be careful about evaluating our functions at single points in time without an enclosing integral). Usually, $X$ will be a space in some Hilbert complex, such as $L^2\Omega^k(M)$ or $H^s\Omega^k(M)$ where the forms are defined over a Riemannian manifold $M$.

We introduce this  framework in order to be able to formulate parabolic problems more generally. It turns out to be useful to consider the concept of \emph{rigged Hilbert space} or \emph{Gelfand triple}, which consists of a triple of separable Banach spaces \[
V \subseteq H \subseteq V^{\ast}
\] such that $V$ is continuously and densely embedded in $H$. For example, if $(V,d)$ is the domain complex of some Hilbert complex $(W,d)$, setting $V = V^{k}$ and $H = W^{k}$ works, as well as various combinations of their products (so that we can use mixed formulations). $H$ is also continuously embedded in $V^\ast$. The standard isomorphism (given by the Riesz representation theorem) between $V$ and $V^\ast$, is not generally the composition of the inclusions, because the primary inner product of importance for weak formulations is the $H$-inner product. It coincides with the notion of distributions acting on test functions. Writing $\aip{\cdot ,\cdot}$ for the inner product on $H$, the setup is designed so that when it happens that some $F \in V^\ast$ is actually in $H$, we have $F(v) = \aip{F,v}$ (which is why we will often write $\aip{F,v}$ to denote the action of $F$ on $v$ even if $F$ is not in $H$). In fact, in most cases of interest, the $H$-inner product is the restriction of a more general bilinear form between two spaces, in which elements of the left (acting) space are of less regularity than elements of $H$, while elements of the right space have more regularity.

Given $A \in C(I,\Lin(V,V^\ast))$, a time-dependent linear operator, we define the bilinear form
\begin{equation}
\label{eq:bilinear}
a(t,u,v) := \aip{-A(t)u,v},
\end{equation}
for $(t,u,v)\in\R\times V\times V$. To proceed, as in elliptic problems, we need $a$ to satisfy some kind of coercivity condition, although it need not be as strong. It turns out that G\r{a}rding's Inequality is the right condition to use here:
\begin{equation}
\label{eq:coercivity}
a(t,u,u)\geq c_1\vn{u}^2_V-c_2\vn{u}^2_H,
\end{equation}
with $c_1$, $c_2$ constants independent of $t\in I$.  Then the following problem is the abstract version of linear, parabolic problems:
\begin{align}
\displaystyle u_t & = A(t)u + f(t)
\label{eq:absparab} \\
\displaystyle u(0) & = u_0.
\label{eq:absparab-ic}
\end{align}
This problem is well-posed:
\begin{theorem}[Existence of Unique Solution to the Abstract Parabolic Problem, \cite{RR2004}, Theorem 11.3]
\label{thm:parabwp}
Let $f\in L^2(I,V^\ast)$ and $u_0\in H$, and~$a$ the time-dependent quadratic form in~\eqref{eq:bilinear}.
Suppose~\eqref{eq:coercivity} holds.
Then the abstract parabolic problem (\ref{eq:absparab}) has a unique solution 
\[u\in L^2( I,V)\cap H^1( I,V^\ast).\]
Moreover, the Sobolev embedding theorem implies $u\in C(I,H)$, which allows us to unambiguously evaluate the solution at time zero, so the initial condition makes sense, and the solution indeed satisfies it: $u(0)=u_0$.
\end{theorem}
This theorem is proved via standard methods (\cite[p. 382]{RR2004}); we take an orthonormal basis of $H$ that is simultaneously orthogonal for $V$ (a frequent situation occurring when, say, it is an orthonormal basis of eigenfunctions of the Laplace operator), formulate the problem in the finite-dimensional subspaces, and use \emph{a priori} bounds on such solutions to extract a weakly convergent subsequence. With this framework, we can show that a wide class of \textsc{PDE} problems, particularly ones that are suited to finite element approximations, are well-posed.
\subsection{Recasting the Problem as an Abstract Evolution Equation}
Let us now see how these results apply in the case of the Hodge heat equation (\ref{eq:par-cnts}) on manifolds. We take a slightly different approach from what is done in \cite{GiHo11a} and \cite{AC2012}, solving an equivalent problem. This sets things up for our modified numerical method detailed in later sections.

Let $(W,d)$ be a closed Hilbert complex, with domain complex $(V,d)$, the standard setup in the above---in particular, we have the Poincar\'e inequality and the well-posedness of the continuous Hodge Laplacian problem. We consider the space $\f Y^{k} := V^{k-1} \times V^{k}$ and its dual $\f Y'=(V^{k-1})'\times (V^{k})'$ with the obvious product norms (we use primes to denote dual spaces so as not to conflict with the dual complex with respect to the Hodge star defined earlier, though these uses are related). This, along with $H = W^{k-1}\times W^{k}$, gives rigged Hilbert space structure
\[
\f Y \subseteq H \subseteq \f Y'.
\]
The embeddings are dense and continuous by definition of the graph inner product and that the operators $d$ have dense domain. We consider the \keyterm{Bochner mixed weak parabolic problem}: to seek a weak solution $(u,\sigma) \in L^2(I,\f Y) \cap H^1 (I, \f Y')$ to the mixed problem
\begin{equation}
\label{eq:boch-mixedweak}
\begin{tabular}{rllll}
$\aip{\sigma,\omega} - \aip{u,d\omega}$ & $= 0,$ & $\forall~\omega\in V^{k-1},$ & $t\in I$,  \\[2mm]
$\aip{u_{t},\varphi} + \aip{du,d\varphi} + \aip{d\sigma,\varphi}$ & $= \aip{f,\varphi},$ & $\forall~ \varphi\in V^k,$ & $t\in I$, \\[2mm]
$u(0)$ & $= g$,
\end{tabular}
\end{equation}
this makes it suitable for approximation using finite-dimensional subspaces of $\f Y '$ (e.g. degrees of freedom for finite element spaces). We see that (\ref{eq:boch-mixedweak}) is the mixed form of (\ref{eq:par-cnts}), which amounts to defining a \emph{system} of differential equations, introducing the variable $\sigma$ defined by $\sigma=d^{*} u$, where $d^{*}$ is the adjoint of the operator $d$. We write the equation weakly (namely, moving $d^{*}$ back to the other side), which makes no difference at the continuous level, but will make a significant difference when discretizing.

In order to use the abstract machinery above, we need a term with $\sigma_t$. Formally differentiating the first equation of \eqref{eq:par-mixedweak}, and substituting $\varphi = d\omega$ in the second equation, we obtain 
\[
0=\aip{\sigma_t, \omega} - \aip{u_t,d\omega} = \aip{\sigma_t, \omega}  - \aip{f,d\omega}+ \aip{d\sigma, d\omega} + \aip{du,dd\omega}.
\]
Since $d^2 = 0$, that last term vanishes, and so, together with the equation for $u_t$, we have the following system:
\begin{equation}
\label{eq:boch-mixedweak-time}
\begin{tabular}{rllll}
$\aip{\sigma_t,\omega}  + \aip{d\sigma, d\omega}$ & $= \aip{f,d\omega},$ & $\forall~\omega\in V^{k-1},$ & $t\in I$,  \\[2mm]
$\aip{u_{t},\varphi} + \aip{d\sigma,\varphi} + \aip{du,d\varphi}$ & $= \aip{f,\varphi},$ & $\forall~ \varphi\in V^k,$ & $t\in I$, \\[2mm]
$u(0)$ & $= g$.
\end{tabular}
\end{equation}
\begin{theorem} Suppose the initial condition $g$ is in the domain of the adjoint $V^{*}$ and $f \in L^{2}(I,(V^{k})')$. Then the problem \eqref{eq:boch-mixedweak-time} is well-posed: there exists a unique solution $(\sigma,u) \in L^2 (I,\f Y  ) \cap H^1 (I,\f Y') \cap C(I, H)$ with $(\sigma(0),u(0)) = (d^{*} g,g)$.
\end{theorem}
\begin{proof}
We see that given $f \in L^2(I, (V^{k})^{\prime})$, we have that the functional $F : (\tau, v) \mapsto \aip{f,d\tau} + \aip{f,v}$ is in $L^2(I, \f Y ')$, since $d$ maps $V^{k-1}$ to $V^k$. For an initial condition on $\sigma$, we can demand that $\sigma(0)$ be the unique $\sigma_0$ statisfying $\aip{\sigma_0,\tau} - \aip{g,d\tau} = 0$. For this to reasonably hold, we must actually have at least $u_0 \in V^{*}_{k}$, the domain of the adjoint operator $d^{*}$, that is, $\sigma_0 = d^{*} g$. We equip the spaces with the standard inner products for product spaces:
\begin{align}
\baip[H]{(\sigma,u ) ,( \tau , v)} & := \aip{\sigma,\tau} + \aip{u,v}\\
\baip[\f Y]{(\sigma,u ) ,( \tau , v)} & := \aip[V] {\sigma,\tau} + \aip[V]{u,v}.
\end{align}
Consider the operator $A : \f Y  \to \f Y '$ defined by 
\[
a(\sigma,u; \omega,\varphi) = \aip[]{-A(\sigma,u), (\omega,\varphi)} = \aip{d\sigma,d\omega} + \aip{d\sigma,\varphi} + \aip{du,d\varphi}.
\]
With the functional $F$ defined as above, we have $F \in L^2 (I , \f Y ')$, and so \eqref{eq:boch-mixedweak-time} is equivalent to the problem
\begin{equation}
(\sigma,u)_t  = A(\sigma,u) + F.
\end{equation}

We now need to verify that the bilinear form $a$ satisfies G\r arding's Inequality:
\begin{align*}
a(\sigma,u;\sigma,u) &=  \vn{d\sigma}^{2} + \aip{d\sigma,u} + \vn{du}^{2} \\
&= \vn{\sigma}^2_V - \vn{\sigma}^{2} + \aip{d\sigma,u} + \vn[V]{u}^{2} - \vn{u}^2 \\
&\geq \vn{\sigma}^2_V - \vn{\sigma}^{2} - \vn{d\sigma}\,\vn{u} + \vn{u}^2_V - \vn{u}^{2}   \\
&\geq  \vn{\sigma}^2_V - \vn{\sigma}^{2} - \frac{1}{2} \vn{\sigma}^2_V - \frac{1}{2}\vn{u}^2_V + \vn{u}^2_V- \vn{u}^{2}  \\
&= \frac{1}{2} \vn{(\sigma,u)}^2_{\f Y} - \vn{(\sigma,u)}_{H}^2.
\end{align*}
Thus, the abstract theory applies, and noting that the initial conditions $(d^{*}g,g) \in H$, we have that \[(\sigma,u) \in L^2 (I,\f Y  ) \cap H^1 (I,\f Y') \cap C(I, H)\] is the unique solution to  \eqref{eq:boch-mixedweak-time} with initial conditions given by $u(0) = g \in V^{*}_{k}$ and $\sigma(0) = d^{*} g$.
\end{proof}
Given this, however, we must still establish that we also have a solution to the original mixed problem (which will be crucial in our error estimates):
\begin{theorem} Let $(\sigma,u)  \in L^2 (I,\f Y  ) \cap H^1 (I,\f Y') \cap C(I, H)$ solve \eqref{eq:boch-mixedweak-time} above with the initial conditions. Then, in fact, $(\sigma,u)$ also solves \eqref{eq:boch-mixedweak}.
\end{theorem}
\begin{proof} The second equation already holds, as it is incorporated unchanged into the equations \eqref{eq:boch-mixedweak-time}. To show the first equation, we show \[
\aip{\sigma_t, \omega} - \aip{u_t,d\omega} = 0
\]
for all time $t$. Then, since the original mixed equation holds at the initial time, standard uniqueness ensures it holds for all $t \in I$. We simply realize it is setting $\varphi = -d\omega$:
\begin{multline*}
\aip{\sigma_t, \omega} - \aip{u_t,d\omega}  = \aip[H]{(\sigma,u)_t, (\omega, -d\omega)} = a(\sigma_t,u_t;\omega, -d\omega) + \aip{f,d\omega} + \aip{f,-d\omega} \\
= \aip{d\sigma, d\omega} + \aip {d\sigma,-d\omega} + \aip{du,dd\omega} = 0.
\end{multline*}
\end{proof}


\section{Error Estimates for the Abstract Parabolic Problem}\label{sec:results-par}
We now combine all the preceding abstract theory (the Holst-Stern~\cite{HoSt10a} framework recalled in \S\ref{sec:approx-hilb-complex}, and the abstract evolution problems framework recalled in \S\ref{sec:abs-evol}) to extend the error estimates of Gillette, Holst, and Zhu~\cite{GiHo11a} and in particular, recover the case of approximating parabolic equations on compact, oriented\footnote{Using differential pseudoforms (\cite[\S2.8]{Fra04}, \cite{We98}), we can eliminate this restriction. However, more theory needs to be developed for that case; the normal projection, in particular. We consider this in future work.}  Riemannian hypersurfaces in $\R^{n+1}$ with triangulations in a tubular neighborhood. The key equation in the derivation of the estimates are the generalizations of Thom\'ee's evolution equations for the error terms. We shall see that these equations lead most naturally to the use of certain Bochner norms for the error estimates that are different for each component in the equation.

Let $(W,d)$ be a closed Hilbert complex with domain $(V,d)$, and the Gelfand triple $\f Y \subseteq H \subseteq \f Y'$ on this complex as above. Now consider our previous standard setup of finite-dimensional approximating complexes $(W_{h},d)$ with domain $(V_{h},d)$, with corresponding spaces $\f Y^{k}_{h} = V^{k-1}_{h} \times V^{k}_{h}$ (it is $\f X^{k}_{h}$ missing the harmonic factor), $i_{h} : V_{h} \hookrightarrow V$ injective morphisms (that are $W$-bounded), $\pi_{h} : V_{h} \to V$ projection morphisms (which may be merely $V$-bounded), and $\pi_{h}\circ i_{h} = \operatorname{id}$. Finally, we consider data interpolation operators $\Pi_{h} : W \to W_{h}$, such that $\Pi_{h} \circ i_{h} =\operatorname{id}$ that realize which projections for the inhomogeneous and prescribed harmonic terms ($f_{h}$ and $w_{h}$ in the abstract theory above) that we use.

\subsection{Discretization of the weak form}\label{rem:semidiscrete}
Suppose we have $f \in L^{2}(I,(V^{k})')$ and $g \in V^\ast_{k}$. Let $(\sigma,u)\in L^2( I, \f Y) \cap H^1( I,\f Y')\cap C(I, H)$ be the unique (continuous) solution to \eqref{eq:boch-mixedweak}, as covered in \S\ref{sec:abs-evol}. As in \cite{GiHo11a}, we can consider approximations to this solution as functionals on finite-dimensional spaces $\f Y_{h}$, e.g. finite element spaces. With the above considerations, we formulate the \keyterm{semi-discrete Bochner parabolic problem}: Find $(\sigma_h,u_h):I\raw \f Y_h$ such that
\begin{equation}
\label{eq:par-semidisc}
\begin{tabular}{rllll}
$\aip[h]{\sigma_h,\omega_h} - \aip[h]{u_h,d\omega_h}$ & $= 0,$ & $\forall~\omega_h\in V^{k-1}_h,$ & $t\in I$\\[2mm]
$\aip[h]{u_{h,t},\varphi_h} + \aip[h]{d\sigma_h,\varphi_h}+\aip[h]{du_h,d\varphi_h}$ & $= \aip[h]{\Pi_h f,\varphi_h},$ & $\forall~ \varphi_h\in V^k_h,$ & $t\in I$ \\[2mm]
$u_h(0)$ & $=g_h$.
\end{tabular}
\end{equation}
(We use the notation of Thom\'ee for the test forms.) We define $g_h$, the projected initial data, shortly. A similar argument as in \S\ref{sec:abs-evol} above, differentiating the first equation with respect to time, considering the Gelfand triple $\f Y^{k}_{h} \subseteq W^{k-1}_{h} \times W^{k}_{h} \subseteq (\f Y^{k}_{h})'$ gives that this problem is well-posed (or more simply, we choose bases and reduce to standard ODE theory as in \eqref{eq:par-semidisc-first} above). Following Gillette, Holst, and Zhu~\cite{GiHo11a}, we define the \keyterm{time-ignorant discrete problem}, using the idea of elliptic projection~\cite{W1973} which we use to define a discrete solution via elliptic projection of the continuous solution at each time $t_0 \in I$: We seek $(\tilde \sigma_h, \tilde u_h,\tilde p_h) \in \f X^{k}_{h}$ such that\index{elliptic projection}
\begin{equation}
\label{eq:par-to-ellip}
\begin{tabular}{rll}
$\aip[h]{\tilde\sigma_h,\omega_h} - \aip[h]{\tilde u_h,d\omega_h}$ & $= 0,$ & $\forall~\omega_h\in V^{k-1}_h$  \\[2mm] 
$\aip[h]{d\tilde \sigma_h,\varphi_h} +\aip[h]{d\tilde u_{h}, d\varphi_{h}}+\aip[h]{\tilde{p}_h, \varphi_h}$ & $= \aip[h]{\Pi_h (-\Delta u(t_0)),\varphi_h},$  & $\forall~ \varphi_h\in V^k_h$  \\[2mm] 
$\aip[h]{\tilde u_h, q_h}$ & $= \aip[h]{\Pi_h (P_{\f H}u(t_0)), q_h}$ & $\forall~q_h \in \mathfrak{H}^k_h$. \\[2mm]
\end{tabular}
\end{equation}
Note that we have included a prescribed harmonic form given by the harmonic part of $u$ (following \cite{AC2012}). We then take the initial data $g_{h}$ to be $\tilde u_{h}(0)$; it is just the solution to the elliptic problem with load data $\Pi_{h}(-\Delta g)$, since $u(0) = g$. Note we do not directly interpolate $g$ itself via $\Pi_h$ for the data; the reason for this will be seen shortly. This discrete problem is well-posed, i.e., a unique solution $u_{h}(t_{0})$ always exists for every time $t_{0} \in I$, by the first part of Theorem \ref{thm:main-estimate-nonzero-harmonic} above. The presence of an additional term $\tilde{p}_{h}$ and equation involving harmonic forms departs from Gillette, Holst, and Zhu~\cite{GiHo11a}, because the theory there is facilitated by the fact that there are no harmonic $n$-forms on open domains in $\R^n$ (the \emph{natural} boundary conditions for such spaces are Dirichlet boundary conditions, in contrast to the more classical example of $0$-forms, i.e. functions). Here, however, we must consider harmonic forms, since we may not be working at the end of an abstract Hilbert complex. For our model problem, namely differential forms on compact orientable manifolds (without boundary), even in the case of $n$-forms, the theory is completely symmetric (by Poincar\'e duality \cite{BoTu82,Jo11,Pet06}).\footnote{Despite this, there are a number of reasons why one should still prefer to continue to phrase problems in terms of $n$-forms if the problem calls for it (\cite{Fra04} describes how it affects the interpretation of certain quantities); and we shall see that it does in fact still make a difference at the discrete level.} In addition, the linear projections $\Pi_{h}$ may not preserve the harmonic space, which gives the possibility of a nonzero $\tilde p_{h}$, despite $-\Delta u$ having zero harmonic part (so it is its own error term). 

\subsection{Determining the error terms and their evolution}
Continuing the method of Thom\'ee~\cite{T2006}, we use the time-ignorant discrete solution as an intermediate reference, and estimate the total errors by comparing to this reference and using the triangle inequality. Roughly speaking, we try to estimate as follows:
\begin{align}
\vn[V]{i_h \sigma_h(t) - \sigma(t)} &\leq \vn[V]{i_h \sigma_h(t) - i_h \tilde \sigma_h(t)} + \vn[V]{i_h \tilde \sigma_h(t) - \sigma(t)} \label{eqn:tot-error-sig}\\
\vn[V]{i_h u_h(t) - u(t)} &\leq \vn[V]{i_h u_h(t) - i_h \tilde u_h(t)} + \vn[V]{i_h \tilde u_h(t) - u(t)}\label{eqn:tot-error-u}.
\end{align}
It turns out that this grouping of the terms is not the most natural for our purposes. We shall see it is the structure of the error evolution equations that groups the terms more naturally as:
\begin{align}
&\vn{i_h u_h(t) - u(t)}\label{eqn:L2-error-u}\\
&\vn{i_{h}\sigma_{h}(t) - \sigma(t)} + \vn{d(i_{h} u_{h}(t) - u(t))}\label{eqn:L2-error-du-sig} \\
&\vn{d(i_{h} \sigma_{h}(t) - \sigma(t))}\label{eqn:L2-error-dsig}.
\end{align}
The sum of these three terms is the sum of the two $V$-norms above. In addition, we shall see in our application to hypersurfaces that this particular grouping of the error terms also corresponds more precisely to the order of approximations in the improved estimates for the elliptic projection (namely, they are of orders $h^{r+1}$, $h^{r}$, and $h^{r-1}$, respectively, for degree-$r$ polynomial differential forms).

The plan is to use the theory of Holst and Stern~\cite{HoSt10a} reviewed in \S\ref{sec:approx-hilb-complex} above to estimate the sum of the two second terms in \eqref{eqn:tot-error-sig} and \eqref{eqn:tot-error-u}; the elliptic projection simply is an approximation, at each fixed time, of the trivial case of $u$ being the solution of the continuous problem with data given by its own Laplacian, $-\Delta u$. The harmonic form portion will come up naturally as part of the calculuation. Using the notation of Thom\'ee~\cite{T2006}, we define the error functions
\begin{align}
\rho(t) &:= \tilde{u}_h(t) - i_{h}^{*}u(t)\label{eq:error-funcs-rho}\\
\theta(t) &:= u_h(t) - \tilde{u}_h(t)\label{eq:error-funcs-theta}\\
\psi(t) &= \sigma_{h}(t) - i_{h}^{*}\sigma(t)\label{eq:error-funcs-psi}.\\
\varepsilon(t) &:= \sigma_h(t) - \tilde{\sigma}_h(t)\label{eq:error-funcs-sig}
\end{align}
(Thom\'ee does not define the third term $\psi$; we have added it for convenience.) In the case that there are no variational crimes (i.e., $J_{h}$ is unitary), the error terms $\rho$ and $\psi$ are bounded above by the elliptic projection errors (because there, $i_{h}^{*}$ is the orthogonal projection, and $\mn{i_{h}^{*}} =\mn{i_{h}} =1$), so that we have, for example, that $\vn{i_{h} u_{h} - u} \leq \vn{\theta} + \vn{\rho} $, corresponding to the use of $\rho$ in \cite{T2006,GiHo11a}. For our purposes, however, the choice of $\rho$ here does not correspond as neatly, now being an intermediate quantity that helps us estimate $\theta$ in terms the elliptic projection error (the second term in \eqref{eqn:tot-error-u}). We find that it contributes more terms with $\vn{I -J_{h}}$. Similar remarks apply for $\sigma$ and $\psi$. We use the method of Thom\'ee to estimate the terms $\theta$ and $\varepsilon$ in terms of (the time derivatives of) $\rho$ and $\psi$, and the elliptic projection error; In order to do this, we need an analogue of Thom\'ee's error equations.
\begin{lemma}[Generalized Thom\'ee error equations]\index{Thom\'ee's error equations} Let $\theta$, $\rho$, and $\varepsilon$ be defined as above. Then for all $t \in I$,
\begin{equation}
\label{eq:par-error}
\begin{tabular}{rll}
$\aip[h]{\varepsilon,\omega_h}- \aip[h]{\theta, d\omega_h}$ &$= 0$ & $\forall \varphi_h \in V_{h}^{k-1}$, \\[2mm]
$\aip[h]{\theta_t,\varphi_h} + \aip[h]{d\varepsilon, \varphi_h} + \aip[h]{d\theta,d\varphi_{h}}$ & $=  \aip[h]{- \rho_t +\tilde{p}_h + (\Pi_{h}-i_{h}^{*})u_{t},\varphi_h}$ & $\forall \omega_h \in V_{h}^{k}$.
\end{tabular}
\end{equation}
\end{lemma}
This differs from Thom\'ee \cite{T2006} and Gillette, Holst, and Zhu \cite{GiHo11a} with the harmonic term $\tilde{p}_h$, which accounts for the projections $\Pi_h$ possibly not sending the harmonic forms to the discrete harmonic forms, an extra $d\theta$ term which accounts for possibly working away from the end of the complex (for differential forms on an $n$-manifold, forms of degree $k<n$), and another data interpolation error term for $u_{t}$ (which also distinguishes it from Arnold and Chen \cite{AC2012}).

\begin{proof} The first equation is simply weakly expressing $\varepsilon$ as $d_{h}^{*}\theta$. This follows immediately from the corresponding equations in the semidiscrete problem and the time-ignorant discrete problem. For the second term, consider the expression 
\begin{equation}\label{eqn:b-thomee-error-derivation}
B:=\aip[h]{\theta_t,\varphi_h} + \aip[h]{d\varepsilon, \varphi_h} + \aip[h]{d\theta,d\varphi_h}  + \aip[h]{\rho_t, \varphi_h},
\end{equation}
and expand it using the definitions to obtain
\begin{multline*}
B=\aip[h]{u_{h,t},\varphi_h}- \aip[h]{\tilde{u}_{h,t},\varphi_h} \\+ \aip[h]{d\sigma_h-d\tilde\sigma_h,\varphi_h} + \aip[h]{du_{h}-d\tilde u_{h},d\varphi}+ \aip[h]{\tilde{u}_{h,t} ,\varphi_h}- \aip[h]{i_{h}^{*} u_t,\varphi_h}.
\end{multline*}
We cancel the $\tilde u_{h,t}$ terms, and apply the semidiscrete equation \eqref{eq:par-semidisc} to cancel the $d\sigma_{h}$ and $du_{h}$ terms, which gives us
\[
B=\aip[h]{\Pi_{h} f,\varphi_{h}} -\aip[h]{d\tilde \sigma_{h},\varphi_{h}} -\aip[h]{d\tilde u_{h},d\varphi_{h}} -\aip[h]{i_{h}^{*}u_{t},\varphi_{h}},
\]
and finally, using the second equation of \eqref{eq:par-to-ellip} to account for the middle terms, we have
\begin{multline*}
B=\aip[h]{\Pi_{h} f,\varphi_{h}} + \aip[h]{\Pi_{h}(\Delta u),\varphi} + \aip[h]{\tilde p_{h},\varphi_{h}} -\aip[h]{i_{h}^{*}u_{t},\varphi_{h}}\\
=\baip[h]{\Pi_{h}\left(\Delta u + f -u_{t}\right),\varphi_{h}} + \aip[h]{\tilde p_{h},\varphi_{h}} + \aip[h]{(\Pi_{h} -i_{h}^{*})u_{t},\varphi_{h}}.
\end{multline*}
But since $u_{t} = \Delta u + f$ is the strong form of the equation, which we know is satisfied by the uniqueness, it follows that $B = \aip[h]{\tilde p_{h}+(\Pi_{h} -i_{h}^{*})u_{t},\varphi_{h}}$. Subtracting the $\rho_{t}$ from both sides  gives the result.
\end{proof}

Now we present our main theorem.
\begin{theorem}[Main parabolic error estimates]\label{eq:main-parabolic-estimate}\index{error estimates!main parabolic estimates theorem}\index{main parabolic estimates theorem|textbf}
Let $(\sigma,u)$ be the solution to the continuous problem \eqref{eq:boch-mixedweak}, $(\sigma_{h},u_h)$ be the semidiscrete solution \eqref{eq:par-semidisc}, $(\tilde \sigma_{h},\tilde u_h)$ the elliptic projection \eqref{eq:par-to-ellip}, and the error quantities \eqref{eq:error-funcs-rho}-\eqref{eq:error-funcs-sig} be defined as above. Then we have the following error estimates:

\begin{align}
\vn[h]{\theta(t)}&\leq \vn[L^{1}(I,W_{h})]{\rho_{t}} + \vn[L^{1}(I,W_{h})]{\tilde p_{h}} + \vn[L^{1}(I,W_{h})] {(\Pi_{h} - i_{h}^{*})u_{t}}\label{eq:main-parabolic-estimate-u}\\
\vn[h]{d\theta(t)} + \vn[h]{\varepsilon(t)} & \leq C\left(\vn[L^{2}(I,W_{h})]{\rho_{t}} + \vn[L^{2}(I,W_{h})]{\tilde p_{h}} + \vn[L^{2}(I,W_{h})] {(\Pi_{h} - i_{h}^{*})u_{t}}\right)\label{eq:main-parabolic-estimate-du-sig}\\
\vn[h]{d\varepsilon(t)} & \leq C\left(\vn[L^{2}(I,W_{h})]{\psi_{t}} + \vn[L^{2}(I,W_{h})]{d^{*}_{h} (\Pi_{h}-i_{h}^{*}) u_{t} }\right),\label{eq:main-parabolic-estimate-dsig}
\end{align}
with
\begin{align}
\vn[L^{2}(I,W_{h})]{\rho_{t}}& \leq C\left( \vn[L^{2}(I,W)]{i_{h} \tilde u_{h,t} - u_{t} } + \mn{I-J_{h}}_{\Lin(W_{h})}\, \vn[L^{2}(I,W)]{u_{t}}\right)\\
\vn[L^{2}(I,W_{h})]{\psi_{t}} &\leq C\left(\vn[L^{2}(I,W)]{i_{h}\tilde \sigma_{h,t} - \sigma_{t}} +\mn{I-J_{h}}_{\Lin(W_{h})}\,\vn[L^{2}(I,W)]{\sigma_{t}}\right).
\end{align}

\end{theorem}
We may further combine these terms, which we shall do in a separate corollary, but it is useful to keep things separate, which allows terms to be analyzed individually when considering specific choices of $V$ and $V_{h}$. The error terms $i_{h} \tilde \sigma_{h} - \sigma$ and $i_{h} \tilde u_{h} - u$ and their time derivatives are furthermore estimated in terms of best approximation norms and variational crimes via the theory of Holst and Stern~\cite{HoSt10a}. The different Bochner norms involved arise from the structure of the error evolution equations.
\begin{proof}
We adapt the proof technique in \cite{T2006, GiHo11a} to our situation, and for ease of notation, unsubscripted norms will denote the $W$-norms and norms subscripted with just $h$ will denote norms on the approximating complex. We now assemble the estimates above separately by computing the $W$-norms of the errors and their differentials. We begin by estimating $\vn[h]{\theta(t)}$. We use the standard technique of using the solutions as their own test functions: Set $\varphi_h = \theta$ and $\omega_h = \varepsilon$ in \eqref{eq:par-error}. Adding the two equations together yields
\begin{equation}
\label{eq:norms-to-ip}
\frac 12 \frac{d}{dt}\vn[h]{\theta}^2 + \vn[h]{\veps}^2 +\vn[h]{d\theta}^{2} = \aip[h]{-\rho_t  +\tilde{p}_{h} +(\Pi_{h} - i_{h}^{*})u_{t},\theta} ,\;\; t\in I
\end{equation}
Following Thom\'ee \cite{T2006}, we introduce $\delta > 0$ to account for non-differentiability at $\theta = 0$, and observe that
\begin{multline*}(\vn[h]{\theta}^2+\delta^2)^{1/2}\frac d{dt}(\vn[h]{\theta}^2+\delta^2)^{1/2}=\frac 12\frac{d}{dt}(\vn[h]{\theta}^2+\delta^2)\\=\frac 12\frac{d}{dt}\vn[h]{\theta}^2 \leq(\vn[h]{\rho_t}  + \vn[h]{\tilde p_h} + \vn[h]{(\Pi_{h}-i_{h}^{*}) u_{t}}) \vn[h]{\theta},
\end{multline*}
using (\ref{eq:norms-to-ip}), the Cauchy-Schwarz inequality, and the definition of operator norms (our goal is to get all of those quantities on the right side of the equation close to zero, so we need not care too much about their sign). Thus, since $\vn[h]{\theta} \leq (\vn[h]{\theta}^2 + \delta^2)^{ 1/2}$, we have, canceling $\vn[h]{\theta}$,
\[
\frac d{dt}(\vn[h]{\theta}^2+\delta^2)^{1/2} \leq \vn[h]{\rho_t}  + \vn[h]{\tilde p_h} + \vn[h]{(\Pi_{h}-i_{h}^{*}) u_{t}}.
\]

Now, using the Fundamental Theorem of Calculus, we integrate from $0$ to $t$ to get
\begin{equation}\label{eqn:theta-est}
\vn[h]{\theta(t)} =  \vn[h]{\theta(0)} + \lim_{\delta \to 0} \int_0^t \frac d{dt}(\vn[h]{\theta}^2+\delta^2)^{1/2}  \leq \int_0^t ( \vn[h]{\rho_t}  + \vn[h]{\tilde p_h} + \vn[h]{(\Pi_{h}-i_{h}^{*}) u_{t}}).
\end{equation}
$\theta(0)$ vanishes by our choice of initial condition as the elliptic projection. 

Next, continuing to follow \cite{GiHo11a}, we consider $\vn[h]{\varepsilon(t)}$. We differentiate the first error equation and substitute $\varphi_{h} = 2\theta_{t}$ and $\omega_{h}=2\varepsilon$, so that
\begin{align}
\aip[h]{\varepsilon_{t},2\varepsilon} - \aip[h]{\theta_{t},2 d\varepsilon} &=0\\
\aip[h]{\theta_{t},2\theta_{t}}+\aip[h]{d\varepsilon,2\theta_{t}}+\aip[h]{d\theta,2d\theta_{t}} &=\aip[h]{-\rho_{t} + \tilde p_{h} + (\Pi_{h}-i_{h}^{*})u_{t},2\theta_{t}}.
\end{align}
Adding the two equations as before, we have, by Cauchy-Schwarz and the AM-GM inequality,
\begin{multline*}
\frac{d}{dt} \vn[h]{\varepsilon}^{2} + 2\vn[h]{\theta_{t}}^{2} + \frac{d}{dt} \vn[h]{d\theta}^{2}\\ \leq 2\vn[h]{\rho_t}\vn[h]{\theta_{t}}  + 2\vn[h]{\tilde p_h}\vn[h]{\theta_{t}} + 2\vn[h]{(\Pi_{h}-i_{h}^{*}) u_{t}}\vn[h]{\theta_{t}}\\ \leq 2\left(\vn[h]{\rho_t}^{2}  + \vn[h]{\tilde p_h}^{2} + \vn[h]{(\Pi_{h}-i_{h}^{*}) u_{t}}^{2}\right) + \tfrac{3}{2}\vn[h]{\theta_{t}}^{2}.
\end{multline*}
Again, dropping some positive terms (this time $\vn[h]{\theta_{t}}^{2}$), using the Fundamental Theorem of Calculus and noting the initial conditions vanish by the choice of elliptic projection, we have
\begin{equation}\label{eqn:eps-squared-d-theta-squared}
\vn[h]{\varepsilon}^{2} + \vn[h]{d\theta}^{2} \leq 2\int_{0}^{t} \left(\vn[h]{\rho_t}^{2}  + \vn[h]{\tilde p_h}^{2} + \vn[h]{(\Pi_{h}-i_{h}^{*}) u_{t}}^{2}\right).
\end{equation}
Finally, we estimate $\vn[h]{d\varepsilon}$. As in the estimate above, we differentiate the first equation with respect to time, and substitute $\omega = 2\varepsilon_{t}$, $\varphi = 2d\varepsilon_{t}$,

\begin{align}
\aip[h]{\varepsilon_{t},2\varepsilon_{t}} - \aip[h]{\theta_{t},2 d\varepsilon_{t}} &=0\\
\aip[h]{\theta_{t},2d\varepsilon_{t}}+\aip[h]{d\varepsilon,2d\varepsilon_{t}}+\aip[h]{d\theta,2dd\varepsilon_{t}} &=\aip[h]{-\rho_{t} + \tilde p_{h} + (\Pi_{h}-i_{h}^{*})u_{t},2d\varepsilon_{t}}.
\end{align}
Noting that $d^{2} = 0$, $\tilde p_{h}$ is perpendicular to the coboundaries, and $\psi = d_{h}^{*} \rho$, we add the equations to get
\begin{multline*}
2\vn[h]{\varepsilon}^{2} + \frac{d}{dt} \vn[h]{d\varepsilon}^{2}  = 2\aip[h]{-\rho_{t} + (\Pi_{h}-i_{h}^{*}) u_{t}, d\varepsilon_{t}} = 2\aip[h]{-\psi_{t}+ d_{h}^{*} (\Pi_{h}-i_{h}^{*}) u_{t}, \varepsilon_{t}}\\
\leq \vn[h]{\psi_{t}}^{2} + \vn[h] {d_{h}^{*} (\Pi_{h}-i_{h}^{*}) u_{t}}^{2} + 2\vn[h]{\varepsilon}^{2}.
\end{multline*}
By the Fundamental Theorem of Calculus, and noting vanishing initial conditions (and an exact cancellation of positive terms), we have
\begin{equation}\label{eqn:d-eps-squared}
\vn[h]{d\varepsilon}^{2} \leq \int_{0}^{t} \left(\vn[h]{\psi_{t}}^{2} + \vn[h] {d_{h}^{*} (\Pi_{h}-i_{h}^{*}) u_{t}}^{2} \right).
\end{equation}

We now estimate $\rho$ and $\psi$. We note that the time derivative of the solutions are also solutions to the mixed formulation, at least provided that $u_{t}$ and other associated quantities are sufficiently regular (in the domain of the Laplace operator) for the norms and derivatives to make sense. Then (recalling $i_{h}^{+} = J_{h}^{-1} i_{h}^{*}$), we have
\begin{multline}
\vn[h]{\rho(t)} = \vn{\tilde u_{h} - i_{h}^{*}u} \leq \vn{\tilde u_{h}- i_{h}^{+} u} +\vn{i_{h}^{+}u- i_{h}^{*} u}\\ \leq\mn{i_{h}^{+}}\left(\vn{i_{h}\tilde u_{h}- u} + \mn{I - J_{h}}\vn{ u}\right),
\end{multline}
and
\begin{multline}
\vn[h]{\psi(t)} = \vn{\tilde \sigma_{h} - i_{h}^{*}\sigma} \leq \vn{\tilde \sigma_{h}- i_{h}^{+} \sigma} +\vn{i_{h}^{+}\sigma- i_{h}^{*} \sigma} \\\leq\mn{i_{h}^{+}}\left(\vn{i_{h}\tilde \sigma_{h}-\sigma} + \mn{I - J_{h}}\,\vn{\sigma}\right).
\end{multline}
The same estimates hold for the time derivatives. The first terms are the estimates that allow us to use the theory of \S\ref{sec:approx-hilb-complex}. We note that the theory acutally uses $V$-norms, but it will work. We cannot improve this in the abstract theory; instead, we use theory for specific choices of $V$, $W$, and $V_{h}$, such as appropriately chosen de Rham complexes and approximations to improve the estimates (\cite[\S3.5]{AFW2010}, \cite[Theorem 3.1]{AC2012}). For these cases, it is helpful to keep the individual estimates on $\vn{\varepsilon}^{2}$, $\vn{\theta}^{2}$, etc. separated. We have combined terms because the abstract theory gives us all the variational crimes together, as it makes heavy use of the bilinear forms above. Additional improvement of estimates based on regularity as done in \cite{AFW2010} cannot made for the variational crimes, as discussed in \cite[\S3.4]{HoSt10a}. We give the relevant example and result in the next section.
\end{proof}
\begin{corollary}[Combined $L^{1}$ estimate] Let $\theta$, $\rho$, $\psi$, and $\varepsilon$ be as above. Then we have
\begin{multline}\label{eq:main-parabolic-estimate-combined}
\vn[L^{1}(I,V)]{i_{h} \sigma_{h} -\sigma} +\vn[L^{1}(I,V)]{i_{h} u_{h} -u}  \leq \\
C\left(\vn[L^{2}(I,W_{h})]{\rho_{t}}  + \vn[L^{2}(I,W_{h})] {(\Pi_{h} - i_{h}^{*})u_{t}} + \vphantom{\vn[L^{2}(I,W_{h})]{\rho_{t}}} \vn[L^{2}(I,W_{h})]{\psi_{t}} + \vn[L^{2}(I,W_{h})]{d^{*}_{h} (\Pi_{h}-i_{h}^{*}) u_{t} } \right. \\ \left. \vphantom{\vn[L^{2}(I,W_{h})]{\rho_{t}}} + \vn[L^{2}(I,W_{h})]{\tilde p_{h}} + \vn[L^{2}(I,V)]{i_{h} \tilde\sigma_{h}-\sigma} + \vn[L^{2}(I,V)]{i_{h}\tilde u_{h} - u}\right).
\end{multline}
Further expanding the time derivative terms, we have
\begin{align*}
\vn[L^{1}(I,V)]{i_{h} \sigma_{h} -\sigma} &+\vn[L^{1}(I,V)]{i_{h} u_{h} -u}  \leq \\
&C\left(\vn[L^{2}(I,W)]{i_{h} \tilde u_{h,t} - u_{t} } +  \vn[L^{2}(I,W)]{i_{h}\tilde \sigma_{h,t} - \sigma_{t}} \right.\\
  &+\mn{I-J_{h}}\, \vn[L^{2}(I,W)]{u_{t}}  +\mn{I-J_{h}}\,\vn[L^{2}(I,W)]{\sigma_{t}} \\
  & +\vn[L^{2}(I,W_{h})] {(\Pi_{h} - i_{h}^{*})u_{t}} + \vn[L^{2}(I,W_{h})]{d^{*}_{h} (\Pi_{h}-i_{h}^{*}) u_{t} }  \\ &\left. + \vn[L^{2}(I,W)]{i_{h} \tilde p_{h}} + \vn[L^{2}(I,V)]{i_{h} \tilde\sigma_{h}-\sigma} + \vn[L^{2}(I,V)]{i_{h}\tilde u_{h} - u}\right).
\end{align*}

\end{corollary}
These terms are organized as follows: the $W$-error in the approximations of the time derivatives, the variational crimes with $\vn{I - J_{h}}$, the data approximation error for the time derivatives, and finally the $V$-approximation errors for the elliptic projection. These can be further expanded in terms of best approximation errors, but we will not have use for that outside of specific examples where the computation is easier done with the previous theorems. This corollary is simply stated for conceptual clarity and a qualitative sense of all the different individual contributions to the error.
\begin{proof}
First, we note that by the Cauchy-Schwarz inequality, the estimate for $\vn{d\theta}$ \eqref{eq:main-parabolic-estimate-u} can be rewritten as using $L^{2}(I,W)$ norms to match the squared terms \eqref{eqn:eps-squared-d-theta-squared} and \eqref{eqn:d-eps-squared}. Combining and absorbing constants, we arrive at
\begin{align*}
\vn[V]{i_{h} \sigma_{h}(t) -\sigma(t)} & +\vn[V]{i_{h} u_{h}(t) -u(t)}  \leq 
C\left(\vn[L^{2}(I,W_{h})]{\rho_{t}}  + \vn[L^{2}(I,W_{h})]{(\Pi_{h} - i_{h}^{*})u_{t}} \right.\\ 
& \left. +  \vn[L^{2}(I,W_{h})]{\psi_{t}} + \vn[L^{2}(I,W_{h})]{d^{*}_{h} (\Pi_{h}-i_{h}^{*}) u_{t}}  + \vn[L^{2}(I,W_{h})]{\tilde p_{h}} \right)  \\
&  + \vn[V]{i_{h} \tilde\sigma_{h}(t)-\sigma(t)} + \vn[V]{i_{h}\tilde u_{h}(t) - u(t)} .
\end{align*}
Integrating from $0$ to $T$, the latter two $V$-norm terms become $L^{1}(I,V)$ norms (and absorb the factor of $T$ from integrating the first into the constant). Finally, using Cauchy-Schwarz to change the $L^{1}(I,V)$ norm into an $L^{2}(I,V)$ norm, and substituting for $\rho_{t}$ and $\psi_{t}$ gives the result.
\end{proof}


\section{Parabolic Equations on Compact Manifolds}\label{sec:par-eqns-riem}
As an application of the preceding results, we return to our original motivating example of de Rham complex to explore an example with the Hodge heat equation on hypersurfaces of Euclidean space, generalizing the discussion in \cite{HoSt10a,GiHo11a}. Let $M$ be compact hypersurface embedded in $\R^{n+1}$. $M$ inherits a Riemannian metric from the Euclidean metric of $\R^{n+1}$.
\subsection{The de Rham Complex on a Manifold}We define the $L^2$ differential $k$-forms on $M$ given by
\[L^2\Omega^k(M):=\left\{\sum_{1\leq i_{1}<\dots<i_{k}\leq n} a_{i_{1}\dots i_{k}} dx^{i_{1}}\wedge \dots \wedge dx^{i_{k}} \in\Omega^k(M)\;:\;a_{i_{1}\dots i_{k}}\in L^2(M)\right\},\]
the standard indexing of differential form basis elements, namely strictly increasing sequences from $\{1,\dots,n\}$. The inner product is given by $\aip{\omega,\eta} = \int \omega \wedge \star \eta$, where $\star$ is the Hodge operator corresponding to the metric.

The weak exterior derivative $d^{k}$ is defined on the domains $\HO k M$, and we have a Hilbert complex $(L^2\Omega,d)$ with domain complex $(\HO {} M,d)$, with $d^{k+1}\circ d^{k} = 0$:
\[
\xymatrix{
0 \ar[r] & H\Omega^0 \ar[r]^-{d^{0}} & H\Omega^1 \ar[r]^-{d^{1}} & \cdots \ar[r]^-{d^{n-1}} & H\Omega^n \ar[r] & 0.
}
\]
As required in the abstract Hilbert complex theory, each domain space carries the graph inner product:
\[\aip[\HO k M]{u,v}:=\aip[L^{2} \Omega^{k}(M)]{u,v}+ \aip[L^{2}\Omega^{k+1}(M)]{d^{k}u,d^{k}v}.\]
For open subsets $U \subseteq \R^{n}$, the ends ($k = 0$ and $k=n$) of this complex are familiar Sobolev spaces of vector fields with the traditional gradient, curl, and divergence operators of vector analysis:
\[
\xymatrix{
0 \ar[r] & H^1(U) \ar[r]^-{\text{grad}} & H(\operatorname{curl}) \ar[r]^-{\text{curl}}& \cdots \ar[r] &H(\operatorname{div}) \ar[r]^-{\Div} & L^2(U) \ar[r] & 0.
}
\]
Similarly, the dual complex is $\HO[*] {}M$ defined by $H^{*} \Omega^{k}(M) := \star \HO {n-k} M$, consisting of Hodge duals of $(n-k)$-forms. We have that the embedding $\HO k M \cap \HO[*] k M \hookrightarrow \e L^{2}\Omega^{k}(M)$ is compact, which enables a Poincar\'e Inequality to hold and the resulting Hilbert complex $(L^{2}\Omega^{k}(M),d)$ to be a closed complex \cite{Pic84,AFW2010}. To summarize, we have the following:
\begin{theorem} Let $M$ be a compact Riemannian hypersurface in $\R^{n+1}$. Then taking $W^{k} = L^{2} \Omega^{k}(M)$, with maps $d^{k}$ the exterior derivative defined on the domains $V^{k} = \HO k M$, $(W,d)$ is a closed Hilbert complex with domain $(V,d)$.
\end{theorem}
We thus are able to define Hodge Laplacians, and see all the abstract theory for the continuous problems \eqref{eqn:mixed-hodge-laplacian-nonzero-harm} and \eqref{eq:boch-mixedweak-time} applies with these choices of spaces. 
\subsection{Approximation of a hypersurface in a tubular neighborhood}
In order to approximate the problems \eqref{eqn:mixed-hodge-laplacian-nonzero-harm} and \eqref{eq:boch-mixedweak-time}, we consider, following \cite{HoSt10a}, a family of approximating hypersurfaces $M_h$ to an oriented hypersurface $M$ all contained in a tubular neighborhood $U$ of $M$. The surfaces $M_{h}$ generally will be piecewise polynomial (say, of degree $s$); the case $s = 1$ corresponds to (piecewise linear) triangulations, studied in \cite{Dziuk88,DeDz06}, and generalized for $s > 1$ in \cite{Demlow2009}. However, the piecewise linear case still is instrumental in the analysis and indeed, the definition of the spaces (via Lagrange interpolation), and so we shall denote it by $T_{h}$ (the triangulation, i.e., set of simplices, will be correspondingly denoted by $\Tau_{h}$, and their images under the interpolation will be denoted $\hat{\Tau}_{h}$). It is convenient, also, to assume that the vertices of the both the triangulation and the higher-degree interpolated surfaces actually lie on the true hypersurface.

The normal vector $\nu$ to the $M$ allows us to define a signed distance function $\dist : U \to \R$ given by \[
\dist(x) = \pm \operatorname{dist}(x,M) = \pm \inf_{y\in M} |x-y|
\]
where the sign is chosen in accordance to which side of the normal $x$ lies on. By elementary theorems in Riemannian geometry \cite[Ch. 6]{doCarmoRG}, $\dist$ is smooth, provided $U$ is small enough; the maximum distance for which it exists is controlled by the sectional curvature of $M$. The normal $\nu$ can be extended to the whole neighborhood; in fact it is the gradient $\nabla \delta$. It is also convenient to define the normals $\nu_{h}$ to the approximating surfaces $M_{h}$. In most of the examples we consider, we assume the vertices of $M_h$ (and $T_{h}$) lie on $M$, but this is not a strict requirement. Instead, we need a condition to ensure that the hypersurfaces $M_{h}$ are diffeomorphic to $M$, eliminating the possibility of a double covering (e.g., as pictured in \cite[Fig. 1, p. 12]{DzEl06}). In particular, we want $M_{h}$ to have the same topology as $M$. This is again restriction on the size of the tubular neighborhood. In such a neighborhood $U$, every $x \in U$ decomposes \emph{uniquely} as
\begin{figure}
\centering
\def\svgwidth{.9\linewidth}
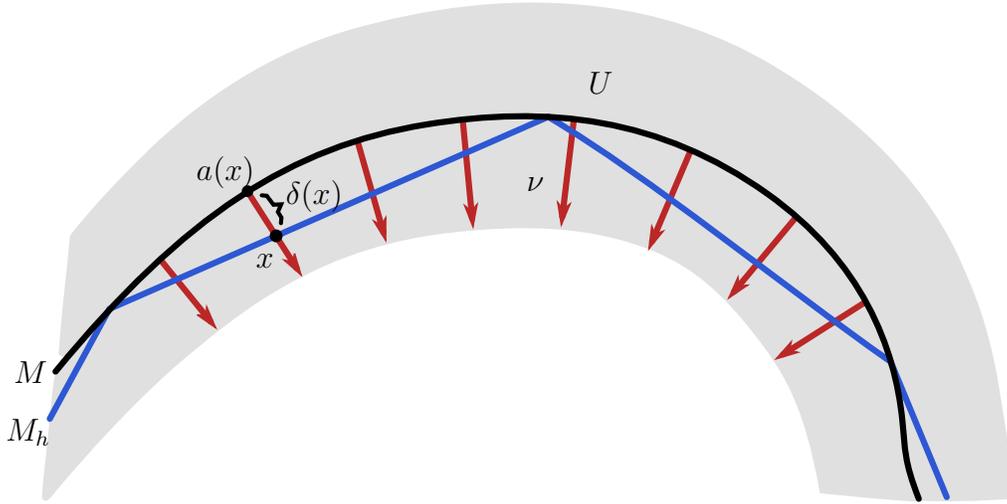
\caption{A curve $M$ with a triangulation (blue polygonal curve $M_{h}$) within a tubular neighborhood $U$ of $M$. Some normal vectors $\nu$ are drawn, in red; the distance function $\delta$ is measured along this normal. The intersection $x$ of the normal with $M_{h}$ defines a mapping $a$ from $x$ to its base point $a(x) \in M$.}\label{fig:tubular-nbhd}
\end{figure}
\begin{equation}
x = a(x) + \dist(x) \nu(x),
\end{equation}
where $a(x) \in M$, and $a : U \to M$ is in fact a smooth function, called the \keyterm{normal projection}. $a$ can then be used to define the degree-$s$ Lagrange interpolated hypersurfaces by considering the image of $T_{h}$ under the degree-$s$ Lagrange interpolation of $a$ over each simplex in $\e T_{h}$ (we write $a_{k} : T_{h} \to M_{h}$ for this) \cite[\S2.3]{Demlow2009}. Now, Holst and Stern~\cite{HoSt10a} show, for hypersurfaces, the following result for the variational crime $\mn{I-J_{h}}$:
\begin{theorem}[Holst and Stern~\cite{HoSt10a}, Theorem 4.4]\label{thm:var-crime-I-Jh-est} Let $M$ be an oriented, compact $m$-dimensional hypersurface in $\R^{m+1}$, and $M_{h}$ be a family of hypersurfaces lying in a tubular neighborhood $U$ of $M$ transverse to its fibers, such that $\vn[\infty]{\delta} \to 0$ and $\vn[\infty]{\nu-\nu_{h}} \to 0$ as $h \to 0$. Then for sufficiently small $h$,
\begin{equation}
\vn{I-J_{h}} \leq C(\vn[\infty] \delta + \vn[\infty]{\nu-\nu_{h}}^{2}).
\end{equation}

\end{theorem}
A result of Demlow~\cite[Proposition 2.3]{Demlow2009} states that, in the case that $M_{h}$ is obtained by degree-$s$ Lagrange interpolation, that $\vn[\infty]{\delta} < Ch^{s+1}$ and $\vn[\infty]{\nu-\nu_{h}} < Ch^{s}$. Thus, putting these results together, we have that
\begin{equation}
\mn{I-J_{h}} \leq Ch^{s+1}.
\end{equation}
Now, the three best approximation error terms \eqref{eqn:main-elliptic-errest} for finite element approximation by polynomials of degree $r$ are bounded by $Ch^{r}$, $Ch^{r+1}$, or $Ch^{r-1}$, depending on the component chosen, so it is crucial to allow for this case, and the convergence rate is optimal when $r=s$. Figure \ref{fig:approx-circle} also dramatically demonstrates how much better a higher-order approximation can be with a given mesh size.\begin{figure}
\centering
\includegraphics[width=0.6\linewidth]{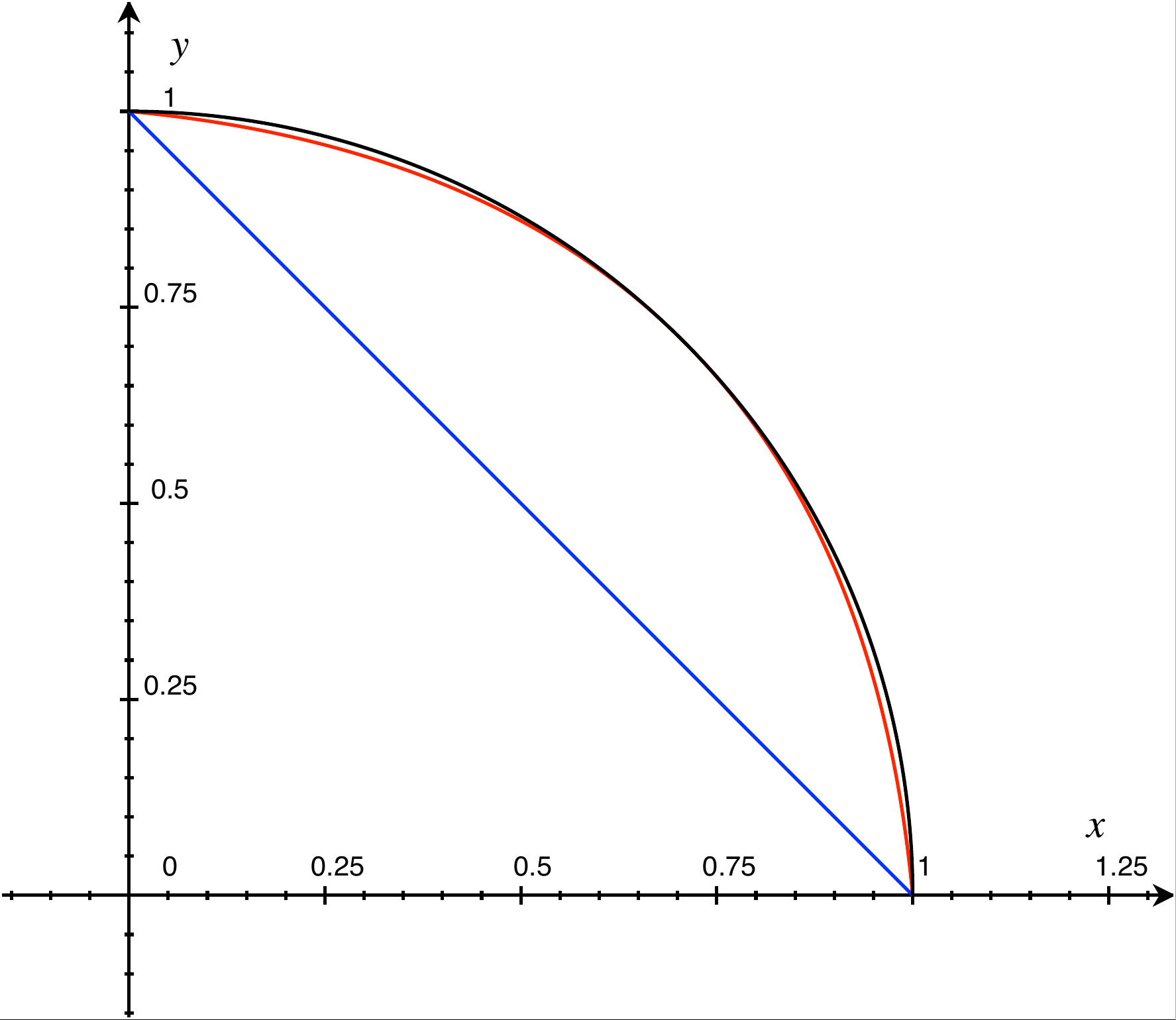}
\caption{Approximation of a quarter unit circle (black) with a segment (blue) and much better quadratic Lagrange interpolation for the normal projection (red), though the underlying triangulation is the same (and thus also the mesh size).}\label{fig:approx-circle}
\end{figure}

Restricting $a$ to the surfaces $M_{h}$ gives diffeomorphisms \[
\rst{a}{M_{h}} : M_{h} \to M.
\]
$a : M_h \to M$ is therefore  a diffeomorphism when restricted to each polyhedron (and is at least globally Lipschitz continuous, the maximum degree of regularity in the piecewise linear case. This is not a problem for Hodge theory, because the form spaces are at most $H^1$ where regularity is concerned; see \cite{Wlok92}). See Figure \ref{fig:tubular-nbhd}.

\subsection{Finite element spaces}
We  choose finite-dimensional subspaces $\Lambda^{k}_h$ of $H\Omega^{k}(M_h)$ for each $k$, satisfying the subcomplex property $d_{h}\Lambda^k_h \subseteq \Lambda^{k+1}_h$. We can then pull forms on $M_h$ back to forms on $M$ via the inverse of the normal projection, which furnishes the injective morphisms $i^k_h : \Lambda^k_h \hookrightarrow H\Omega^{k}(M)$ (since pullbacks commute with $d$) required by the theory above in Section \ref{sec:feec}. The main finite element spaces relevant for our purposes are two families of piecewise polynomials, discussed in detail in \cite{AFW2006,AFW2010}. We must choose these spaces for our equations in a specific relationship in order for the numerical methods and theory detailed above to apply, and for the approximations to work. This is why we prefer a piecewise polynomial approximation of $M$  as opposed to a curved triangulation of $M$ itself; these are shown to have these necessary properties.  
\begin{definition}[Polynomial differential forms]\index{differential form!polynomial finite elements}
Let $\poly_r$ denote polynomials of degree at most $r$, in $n$ variables, and $\mathcal H_r$ be the subspace of homogeneous polynomials. We define the first family, denoted $\poly_r\Lambda^k(\Tau)$, to consist of all $k$-forms with coefficients belonging to $\poly_r$ when restricted to each $n$-simplex of $\Tau$. The continuity condition is that the polynomials on two simplices having a common face must have the same trace to that face. The second family, denoted $\poly_r^-\Lambda^k(\Tau)$, are intermediate spaces, between the spaces of the first class:
\[\poly_{r-1}\Lambda^k(\Tau)\subsetneq \poly_r^-\Lambda^k(\Tau) \subsetneq \poly_r\Lambda^k(\Tau).\]
These are defined as follows: first, consider the radial vector field $X = x^{i} \frac{\partial}{\partial x^{i}}$, that is, at each $x$, it is a radially pointing vector of length $|x|$, and then define the \keyterm[Koszul differential]{Koszul operator} $\kappa \omega := X \lrcorner\, \omega$, the interior product with $X$. Then \[
\poly_r^-\Lambda^k(\Tau):=\poly_{r-1}\Lambda^k\oplus\kappa\mathcal H_{r-1}\Lambda^{k+1}.
\] This is a direct sum, since $\kappa$ always raises polynomial degree and decreases form degree, so yields homogeneous polynomials of degree $r$. $\kappa$ is in some ways dual to the operator $d$ (which, in particular, increases form degree and decreases polynomial degree), and by the properties of interior products, $\kappa^{2} = 0$.
\end{definition}

These polynomial spaces generalize existing finite element spaces, such as Whitney forms, N\'ed\'elec elements, and Raviart-Thomas elements (see \cite{GiHo11a,AFW2010} for these examples and more), realizing the collection and clarification of previous results respecting vector methods, as we have mentioned numerous times throughout this work.

The important property of these spaces is that they admit the cochain projections whose role we have seen is so important in the theory. First, we describe the case where $M = U$ is a domain in $\R^{n}$ with smooth or Lipschitz boundary.
\begin{equation}
\label{eq:coch-proj-def}
\pi_h^k:L^2\Omega^k\raw  \Lambda^k_h\quad\text{where~$\Lambda^k_h\in\{\poly_r\Lambda^k(\Tau),~\poly_r^-\Lambda^k(\Tau)\}$}.
\end{equation}
These operators, by virtue of their construction, are uniformly bounded (in $L^2\Omega^k$, not just $H\Omega^{k}$) with respect to $h$. 
Finally, the following theorem explicitly expresses the projection error (and hence, best approximation error) in terms of powers of the mesh size $h$ and the norms of the solution. \index{bounded cochain projection!for open subsets of Rn@for open subsets of $\R^{n}$}
\begin{theorem}[Arnold, Falk, and Winther \cite{AFW2010}, Theorem 5.9]\hfill
\label{thm:afw-5pt9} 
\renewcommand{\labelenumi}{(\roman{enumi}.)}
\begin{enumerate}
\item Let $\Lambda^k_h$ be one of the spaces $\poly^-_{r+1}\Lambda^k(\Tau)$ or, if $r\geq 1$, $\poly_r\Lambda^k(\Tau)$.  Then $\pi^k_h$ is a cochain projection onto $\Lambda^k_h$ and satisfies
\[\vn{\omega-\pi^k_h\omega}_{L^2\Omega^k(U)}\leq ch^s\vn{\omega}_{H^s\Omega^k(U)},\quad \omega\in H^s\Omega^k(U),\]
for $0\leq s\leq r+1$.  Moreover, for all $\omega\in L^2\Omega^k(U)$, $\pi^k_h\omega\raw\omega$ in $L^2$ as $h\raw 0$.
\item Let $\Lambda^k_h$ be one of the spaces $\poly_r\Lambda^k(\Tau)$ or $\poly^-_r\Lambda^k(\Tau)$ with $r\geq 1$.  Then 
\[\vn{d(\omega-\pi^k_h\omega)}_{L^2\Omega^k(U)}\leq ch^s\vn{d\omega}_{H^s\Omega^k(U)},\quad \omega\in H^s\Omega^k(U),\]
for $0\leq s\leq r$.
\end{enumerate}
\end{theorem}
These bounded cochain operators are explicitly constructed in \cite{AFW2006,AFW2010}; they are the natural interpolation operators defined for continuous differential forms and analogous to polynomial interpolation operators on functions, but combined with smoothings to allow extension to $H^{s}$ differential forms which may not necessarily be continuous.

\begin{example}[The Mixed Hodge Laplacian problem on an open subset of $\R^{n}$]\index{finite element spaces!for domains in Rn@for domains in $\R^{n}$}
For the mixed Hodge Laplacian problem we considered above, we must choose $\Lambda_{h}^{k-1}$ and $\Lambda_{h}^{k}$ in such a manner such that $d\Lambda_{h}^{k-1} \subseteq \Lambda_{h}^{k}$; one cannot make the choices of spaces completely independent of one another for our mixed problem \cite[\S5.2]{AFW2010}. For example, if we choose $\Lambda_{h}^{k-1} = \poly_{r}\Lambda^{k-1}(\Tau_{h})$, we necessarily must choose \[\Lambda_{h}^{k} \in \left\{\poly_{r-1}\Lambda^{k}(\Tau_{h}),\ \poly_{r}^{-}\Lambda^{k}(\Tau_{h})\right\}.\] Similarly, for $\Lambda_{h}^{k-1} = \poly^{-}_{r}\Lambda^{k-1}(\Tau_{h})$, we choose \[
\Lambda_{h}^{k} \in \left\{ \poly_{r}^{-} \Lambda^{k}(\Tau_{h}),\ \poly_{r-1} \Lambda^{k}(\Tau_{h})\right\}.
\] Continuing in this manner down the complex, there are $2^{n}$ possible full cochain subcomplexes one can form with these choices of spaces. Of course, for one single Hodge Laplacian problem, we only need to work with three spaces in the chain, since the equations only involve $(k-1)$- and $k$-forms and their differentials.
\end{example}

\begin{example}[Finite Element Spaces on Riemannian manifolds]\index{finite element spaces!for Riemannian manifolds}
Now, suppose we are back in the situation with a Riemannian hypersurface $M \subseteq \R^{n+1}$, with a family of degree-$s$ Lagrange-interpolated surfaces $M_{h}$, over a triangulation $T_{h}$. We can still consider the polynomial finite element spaces on the triangulation $\Tau_{h}$ as before; the only difference here is that the simplices may not join up smoothly (i.e., as a manifold, it may have corners). This is not a problem, because the continuity conditions enforced by the finite element spaces also allow for discontinuities or non-classical-differentiability on the simplicial boundary faces. To define the analogous polynomial spaces on the possibly curved triangulations $M_{h}$, we simply say a form is in the analogous polynomial spaces $\poly_{r} \Lambda^{k}(\hat{\Tau}_{h})$ if its pullback by the inverse of the interpolated normal projections $a_{k} : T_{h} \to M_{h}$ to $T_{h}$ is in $\poly_{r}\Lambda^{k}(\Tau_{h})$ \cite[\S2.5]{Demlow2009}. Now, from $\poly_{r}\Lambda^{k}(\hat \Tau_{h})$, we pull these forms back to the surface $M$ via the normal projections $(\rst{a}{M_{h}})^{-1}$. This gives the injective morphisms $i_{h}^{k} : \Lambda_{h}^{k} \to \HO k M$; it commutes with the differentials, since the pullbacks do.

For the bounded cochain operators, the situation is similar. We have $\pi^{\prime k}_{h} : \HO k {M_{h}} \to \Lambda_{h}^{k}$ a cochain projection defined by pulling forms defined in neighborhoods back to the triangulations (using the trace theorem if necessary), as constructed in \cite{AFW2006,AFW2010}. Then we compose with the pullbacks $(\rst{a}{M_{h}})^{*}$. This gives us the cochain projections $\pi_{h}^{k} : H \Omega^{k}( M) \to \Lambda_{h}^{k}$ (by  \cite[Theorem 3.7]{HoSt10a}).
\end{example}
\subsection{Estimates for the Mixed Hodge Laplacian problem on manifolds}\index{error estimates!for the parabolic Hodge Laplacian problem}
With this, we can then integrate the terms from \cite[Example 4.6]{HoSt10a} to get the results for the parabolic equations (or, equivalently, add the variational crimes to \cite{GiHo11a,AC2012}). Let us consider now the mixed Hodge Laplacian problem on Riemannian hypersurfaces, considering the setup in the previous example. Namely, we consider $W^{k} = \e L^{2}\Omega^{k}(M)$, $V^{k} = \HO k M$ as above, the approximating spaces  $V_{h}^{k-1} = \poly_{r+1} \Lambda^{k-1}(\hat{\e T}_{h})$ and $V_{h}^{k} = \poly_{r}\Lambda^{k}(\hat{\e T}_{h})$, and finally the inclusion and projection morphisms as above (possibly with additional pullbacks for interpolation degree $s > 1$). Of course, as mentioned before, these are not the only ways of choosing the spaces, but we stay with, and make estimates based on, this choice for the remainder of this example (the same choice made in \cite[Example 4.6]{HoSt10a}). For a function $\tilde{f} \in L^{2} \Omega^{k}(M)$, we have an approximate solution $(\sigma'_{h},u'_{h},p'_{h}) \in i_{h} \f X'_{h}$ to the elliptic problem, on the true subcomplex $i_{h} W_{h}$ (with modified inner product, as in the theory of \S\ref{sec:our-extension}). For $\tilde f$ sufficiently regular, and $(\sigma,u,p)$ satisfying the regularity estimate \cite{AFW2010,GiHo11a}
\begin{equation}\label{eqn:regularity-estimate}
\vn[H^{s+2}]{u} +\vn[H^{s+2}]{p} + \vn[H^{s+1}]{\sigma} \leq C\vn[H^{s}]{\tilde f},
\end{equation}
for $0 \leq s \leq s_{\max}$, then, since we are in the de Rham complex, where the cochain projections are $W$-bounded, we have the improved error estimates of Arnold, Falk, and Winther~\cite[\S3.5 and p. 342]{AFW2010} for the elliptic problem:
\begin{align}
\vn{ u - i_{h}u_{h}'} +\vn{p-i_{h}p_{h}'} &\leq Ch^{r+1}\vn[H^{r-1}]{\tilde f} \label{eqn:improved-u}\\
\vn{ d(u - i_{h}u_{h}')} +\vn{\sigma-i_{h}\sigma_{h}'} &\leq Ch^{r}\vn[H^{r-1}]{\tilde f} \label{eqn:improved-du-sig}\\
\vn{ d(\sigma - i_{h}\sigma_{h}')} &\leq Ch^{r-1}\vn[H^{r-1}]{\tilde f}.
\end{align}
We should also note that Arnold and Chen~\cite{AC2012} prove that this also works for a nonzero harmonic part \cite[Theorem 3.1]{AC2012}.
Holst and Stern~\cite{HoSt10a} augment these estimates to include the variational crimes, so that (changing the notation to suit our problem) for $(\tilde \sigma_{h} , \tilde u_{h}, \tilde p_{h}) \in \f X_{h}$, the discrete solution to the elliptic problem now on the approximating complexes we have chosen, we have the estimates
\begin{multline}\label{eqn:combined-vcs-estimate}
\vn{u- i_{h}\tilde u_{h}} + \vn{p- i_{h}\tilde p_{h}} + h\left(\vn{d(u-i_{h}\tilde u_{h})} + \vn{\sigma - i_{h}\tilde\sigma_{h}} \right) \\
+ h^{2} \vn{d(\sigma - i_{h}\tilde\sigma_{h})} \leq C( h^{r+1}\vn[H^{r-1}]{\tilde f} + h^{s+1}\vn{\tilde f}).
\end{multline}
We note the terms associated to the different powers of $h$ above correspond exactly to the breakdown \eqref{eq:main-parabolic-estimate-u}-\eqref{eq:main-parabolic-estimate-dsig} above. For the elliptic projection in our problem, we also need to account for the nonzero harmonic part of the solution. Setting $\tilde w = P_{\f H} \tilde u$ and $\tilde w_{h} = \Pi_{h} \tilde w$, we have that our three additional terms (given by Theorem \ref{thm:main-estimate-nonzero-harmonic} above) are the corresponding best approximation error $\inf_{v \in V_{h}^{k}}\vn[V]{\tilde w-v}$,  the $\mn{I-J_{h}}$ term, and the data approximation $\vn[h]{\tilde w_{h} - i_{h}^{*}\tilde w}$. For the best approximation, we make use of our observation about the inequality \eqref{eqn:I-PH-to-inf}, in which we may instead use the $W$-norm instead of the $V$-norm in the case that the projections are $W$-bounded, as they are here in the de Rham complex. Because $\tilde w$ is harmonic, it is smooth (and in particular, in $H^{r+1}$), so we may apply Theorem \ref{thm:afw-5pt9} to find that it is of order $Ch^{r+1}\vn[H^{r+1}]{\tilde w}$. The $\mn{I-J_{h}}$ term has already been shown to be of order $Ch^{s+1}$ above in Theorem \ref{thm:var-crime-I-Jh-est}. Finally, by Theorem \ref{thm:fam-of-projections} above, we have that data approximation splits into the other two terms. Therefore, to summarize, we have
\begin{theorem}[Estimates for the elliptic projection]\index{error estimates!for the elliptic projection} Consider $(\sigma(t),u(t))$, the solution to the parabolic problem \eqref{eq:boch-mixedweak-time} and $(\sigma_{h}(t),u_{h}(t))$ the semidiscrete solution in \eqref{eq:par-semidisc} above. Then we have the following estimates for the elliptic projection $(\tilde \sigma_{h},\tilde u_{h},\tilde p_{h})$:
\begin{multline}\label{eqn:combined-vcs-estimate-elliptic-proj}
\vn{u- i_{h}\tilde u_{h}} + \vn{i_{h}\tilde p_{h}} + h\left(\vn{d(u-i_{h}\tilde u_{h})} + \vn{\sigma - i_{h}\tilde\sigma_{h}} \right) \\
+ h^{2} \vn{d(\sigma - i_{h}\tilde\sigma_{h})} \leq C\left( h^{r+1}\left(\vn[H^{r-1}]{\Delta u}+\vn[H^{r+1}]{\tilde w}\right) + h^{s+1}\left(\vn{\Delta u} + \vn{\tilde w}\right)\right).
\end{multline}

\end{theorem}
(We note $p =  P_{\f H}(-\Delta u)  = 0$.) We now would like use the our main parabolic estimates to analyze the analogous quantity
\begin{equation}\label{eqn:combined-vcs-estimate-evol}
\vn{u(t)- i_{h} u_{h}(t)}  + h\left(\vn{d(u(t)-i_{h} u_{h}(t))} + \vn{\sigma(t) - i_{h}\sigma_{h}(t)} \right) 
+ h^{2} \vn{d(\sigma(t) - i_{h}\sigma_{h}(t))},
\end{equation}
and its integral, i.e. Bochner $L^{1}$ norm.
\begin{theorem}[Main combined error estimates for Riemannian hypersurfaces] Let $(\sigma(t),u(t))$, $(\sigma_{h}(t),u_{h}(t))$, and all terms involving the elliptic projection are defined as above, and the regularity estimate \eqref{eqn:regularity-estimate} is satisfied. Then
\begin{multline*}
\vn[L^{1}(W)]{u- i_{h} u_{h}}  \\ + h\left(\vn[L^{1}(W)]{d(u-i_{h} u_{h})}  + \vn[L^{1}(W)]{\sigma - i_{h}\sigma_{h}} \right) 
+ h^{2} \vn[L^{1}(W)]{d(\sigma - i_{h}\sigma_{h})} \\ \leq C\left[h^{r+1}\left((T+1)\left(\vn[L^{1}(H^{r-1})]{\Delta u}+\vn[L^{1}(H^{r+1})]{\tilde w}\right) + T\left(\vn[L^{1}(H^{r-1})]{\Delta u_{t}} +\vn[L^{1}(H^{r+1})]{\tilde w_{t}}\right) \right) \right.\\+ \left. h^{s+1}\left((T+1)\left(\vn[L^{1}(W)]{\Delta u}  +\vn[L^{1}(W)]{\tilde w}\right) + T\left(\vn[L^{1}(W)]{\Delta u_{t}} +\vn[L^{1}(W)]{\tilde w_{t}}\right) \right)\right].
\end{multline*}
\end{theorem}
(We abbreviate $L^{p}(I,X)$ as $L^{p}(X)$.) The constants $T$, of course, can be further rolled into the constant $C$. We remark that in previous results, factors of $T$ show up on the $\vn{\Delta u_{t}}$ terms, and, heuristically speaking, this is due to the $u_{t}$ being a physically different quantity, namely, a rate of change. However, the appearance of the factor of $T$ on the  $\vn{\Delta u}$ comes from the harmonic approximation error $\tilde p_{h}$, which is, physically speaking, a harmonic source term. The details depend on the nature of the approximation operators $\Pi_{h}$.
\begin{proof}
By the triangle inequality, we have that \eqref{eqn:combined-vcs-estimate-evol} breaks up into something of the form \eqref{eqn:combined-vcs-estimate} (taking $(\tilde \sigma_{h}, \tilde u_{h},\tilde p _{h})$ to be elliptic projection with $\tilde f = -\Delta u(t)$ and $\tilde p =0$; here $\tilde f$ is not to be confused with the \emph{parabolic} source term $f(t)$) and
\begin{equation}
\mn{i_{h}}\, \left(\vn[h]{\theta(t)} + h(\vn[h]{\varepsilon(t)} + \vn[h]{du(t)})+h^{2}\vn[h]{d\varepsilon(t)} \right),
\end{equation}
recalling the error quantities defined in \eqref{eq:error-funcs-rho}-\eqref{eq:error-funcs-sig}. Now, substituting our estimates \eqref{eq:main-parabolic-estimate-u}-\eqref{eq:main-parabolic-estimate-dsig}, we then have
\begin{multline}
\vn[h]{\theta(t)}  \leq \vn[L^{1}(W_{h})]{\rho_{t}} + \vn[L^{1}(W_{h})]{\tilde p_{h}} + \vn[L^{1}(W_{h})] {(\Pi_{h} - i_{h}^{*})u_{t}}\\
\leq C\left(\vn[L^{1}(W)]{i_{h} \tilde u_{h,t} - u_{t}} +   \vn[L^{1}(W_{h})]{\tilde p_{h}} +\vn{I-J_{h}}\, \vn[L^{1}(W)]{u_{t}}  + \vn[L^{1}(W_{h})] {(\Pi_{h} - i_{h}^{*})u_{t}}\right) \\
\leq C_{1} h^{r+1} \left(\vn[L^{1}(H^{r-1})] {\Delta u}  +\vn[L^{1}(H^{r-1})] {\Delta u_{t}} +\vn[L^{1}(H^{r+1})]{\tilde w}+\vn[L^{1}(H^{r+1})]{\tilde w_{t}}\right) \\ + C_{2}h^{s+1} \left(\vn[L^{1}(W)]{\Delta u} + \vn[L^{1}(W)]{\Delta u_{t}}+\vn[L^{1}(W)]{\tilde w}+\vn[L^{1}(W)]{\tilde w_{t}}\right).
\label{eq:main-parabolic-estimate-u-mfld}
\end{multline}
For $\vn[h]{d\theta} + \vn[h]{\varepsilon}$, the computation is almost exactly the same, except with possibly different constants, to account for using $L^{2}$ Bochner norms, and that :
\begin{multline*}
\vn[h]{d\theta(t)} + \vn[h]{\varepsilon(t)} \\
 \leq C\left( \vn[L^{2}(W)]{i_{h} \tilde u_{h,t} - u_{t} } +   \vn[L^{2}(W_{h})]{\tilde p_{h}}+ \mn{I-J_{h}}\, \vn[L^{2}(W)]{u_{t}} +  \vn[L^{2}(W_{h})] {(\Pi_{h} - i_{h}^{*})u_{t}}\right)\\
C_{3} h^{r+1} \left(\vn[L^{2}(H^{r-1})] {\Delta u}  +\vn[L^{2}(H^{r-1})] {\Delta u_{t}} +\vn[L^{2}(H^{r+1})]{\tilde w}+\vn[L^{2}(H^{r+1})]{\tilde w_{t}}\right)\\ + C_{4}h^{s+1} \left(\vn[L^{2}(W)]{\Delta u} + \vn[L^{2}(W)]{\Delta u_{t}} +\vn[L^{2}(W)]{\tilde w}+\vn[L^{2}(W)]{\tilde w_{t}}\right).
\end{multline*}
These terms are actually absorbed into the lower order terms by the extra factor of $h$, due to consisting entirely of the same order terms except using a different norm. However, the situation is slightly different for $\vn[h]{d\varepsilon}$; namely we use \eqref{eqn:improved-du-sig} to get a term of order $h^{r}$, and the $d_{h}^{*}$ on the variational crime part also removing a factor of $h$:
\begin{multline*}
\vn[h]{d\varepsilon(t)} \leq C\left(\vn[L^{2}(W_{h})]{\psi_{t}} + \vn[L^{2}(W_{h})]{d^{*}_{h} (\Pi_{h}-i_{h}^{*}) u_{t} }\right) \\
 \leq C\left(\vn[L^{2}(W)]{i_{h}\tilde \sigma_{h,t} - \sigma_{t}} +\mn{I-J_{h}}\,\vn[L^{2}(W)]{\sigma_{t}} +\vn[L^{2}(W_{h})]{d^{*}_{h} (\Pi_{h}-i_{h}^{*}) u_{t} }\right)\\
 \leq C_{5} h^{r} (\vn[L^{2}(H^{r-1})] {\Delta u_{t}} +\vn[L^{2}(H^{r+1})]{\tilde w}) \\
 + C_{6}h^{s} (\vn[L^{2}(W)]{\Delta u} + \vn[L^{2}(W)]{\Delta u_{t}}  +\vn[L^{2}(W)]{\tilde w}+\vn[L^{2}(W)]{\tilde w_{t}})
\end{multline*}
However, we see that multiplying by $h^{2}$, this term also gets absorbed; thus we need only consider the error from $\vn[h]{d\theta}$ in further calculation of the combined estimate. We have, thus far:
\begin{multline} \label{eqn:combined-pointwise}
\vn{u(t)- i_{h} u_{h}(t)}  \\+ h\left(\vn{d(u(t)-i_{h} u_{h}(t))} + \vn{\sigma(t) - i_{h}\sigma_{h}(t)} \right) 
+ h^{2} \vn{d(\sigma(t) - i_{h}\sigma_{h}(t))} \\
\leq C_{1} h^{r+1} \left(\vn[L^{1}(H^{r-1})] {\Delta u}  +\vn[L^{1}(H^{r-1})] {\Delta u_{t}} +\vn[L^{1}(H^{r+1})]{\tilde w}+\vn[L^{1}(H^{r+1})]{\tilde w_{t}} \right) \\+ C_{2}h^{s+1} (\vn[L^{1}(W)]{\Delta u} + \vn[L^{1}(W)]{\Delta u_{t}} +\vn[L^{1}(W)]{\tilde w}+\vn[L^{1}(W)]{\tilde w_{t}}) \\+ C\left( h^{r+1}\left(\vn[H^{r-1}]{\Delta u(t)} +\vn[H^{r+1}]{\tilde w(t)}\right) + h^{s+1}\left(\vn{\Delta u(t)} +\vn{\tilde w(t)}\right)\right).
\end{multline}
Integrating with respect to $t$ from $0$ to $T$, we find that the already-present Bochner norms are constant and thus introduce an extra factor of $T$. Absorbing the constants except $T$ gives the result.
\end{proof}
 This shows, in particular, that the optimal rate of convergence occurs when $r=s$, i.e., the polynomial degree of the finite element functions matches the degree of polynomials used to approximate the hypersurface. This tells us, for example, it is not beneficial to use higher-order finite elements on, say, a piecewise linear triangulation. 
Finally, to put these estimates into some perspective and help develop some intuition for their meaning, we present the generalization of the estimates of Thom\'ee from the introduction.
\begin{corollary}[Generalization of \cite{T2006,GiHo11a,AC2012}] Focusing on just the components $u$ and $\sigma$ separately, we have the following estimates (assuming the regularity estimates \eqref{eqn:regularity-estimate} are satisfied), and supposing $r= s$, i.e., the finite element spaces considered consist of polynomials of the same degree as the interpolation on the surface:
\begin{equation*}
 \vn{u(t) - i_{h}u_{h}(t)} \leq Ch^{r+1}\left(\vn[H^{r+1}]{u(t)}  \vphantom{\int_{0}^{t}}+ \int_{0}^{t} \left(\vn[H^{r+1}]{u(s)} + \vn[H^{r+1}]{u_{t}(s)}\right)\; ds\right)
\end{equation*}
\[
\vn{\sigma(t) -i_{h}\sigma_{h}(t)} \leq  Ch^{r+1}\left(\vn[H^{r+2}]{u(t)} + \left(\int_{0}^{t} \left(\vn[H^{r+1}]{u(s)}^{2} + \vn[H^{r+1}]{u_{t}(s)}^{2}\right)\,ds\right)^{1/2}\right)
\]
\end{corollary}
This easily leads to an estimate in a Bochner $L^{\infty}$ norm (simply take the sup in the non-Bochner norm terms and $t=T$ in the integrals); this shows that the error in time is small at every $t \in I$. Similar estimates hold for $L^{2}(I,W)$ norms.
\begin{proof} We consider the improved error estimate and variational crimes in $u$ and $\sigma$ separately. We first have, by expanding the terms in \eqref{eqn:theta-est} as in the derivation of \eqref{eq:main-parabolic-estimate-u-mfld},
\begin{multline*}
\vn{u(t) - i_{h}u_{h}(t)} \leq C\left( \vn{u(t) -i_{h}\tilde u_{h}(t)}+ \vn{\theta(t)}\right)  \\\leq
Ch^{r+1}\left(\vn[H^{r-1}]{\Delta u(t)} + \vn[H^{r+1}]{\tilde w(t)} \vphantom{\int_{0}^{t}} \right. \\
\left.+ \int_{0}^{t} \left(\vn[H^{r-1}]{\Delta u(s)} + \vn[H^{r-1}]{\Delta u_{t}(s)}+\vn[H^{r-1}]{\tilde w(s)}+\vn[H^{r-1}]{\tilde w_{t}(s)}\right) \; ds\right).\end{multline*}
The result follows by noting that $\vn[H^{r+1}]{u}$ includes estimates on all the second derivative terms in $u$, and $\tilde w= P_{\f H} u$, so those two norms can all be combined (with possibly different constants). Next, we consider $\sigma$. The improved error estimates \cite[p. 342]{AFW2010} imply that if we do not combine estimates involving $du$ with those of $\sigma$ for the modified solution, and $\tilde f$ is regular enough to use the $H^{r}$- rather than $H^{r-1}$-norm, then we can gain back one factor of $h$, so that it is of order $h^{r+1}$ (rather than $h^{r}$ as in \eqref{eqn:improved-du-sig}). On the other hand, the elliptic projection error $\vn{\varepsilon(t)}$ still can be taken along with $\vn{d\sigma(t)}$ and was of order $h^{r+1}$ to begin with. Thus, applying \eqref{eqn:eps-squared-d-theta-squared}, we have
\begin{multline*}
\vn{\sigma(t) - i_{h}\sigma_{h}(t)} \leq C\left( \vn{\sigma(t) -i_{h}\tilde \sigma_{h}(t)} +\vn{\varepsilon(t)} + \vn{du(t)}\right)  \\\leq
Ch^{r+1}\left(\vn[H^{r}]{\Delta u(t)} + \vn[H^{r+1}]{\tilde w(t)} \vphantom{\int_{0}^{t}} \right. \\
\left.+ \left[\int_{0}^{t} \left(\vn[H^{r-1}]{\Delta u(s)}^{2} + \vn[H^{r-1}]{\Delta u_{t}(s)}^{2}+\vn[H^{r-1}]{\tilde w(s)}^{2}+\vn[H^{r-1}]{\tilde w_{t}(s)}^{2}\right) \; ds\right]^{1/2}\right)\\
\leq Ch^{r+1}\left(\vn[H^{r+2}]{u(t)} + \left(\int_{0}^{t} \left(\vn[H^{r+1}]{u(s)}^{2} + \vn[H^{r+1}]{u_{t}(s)}^{2}\right)\,ds\right)^{1/2}\right),
\end{multline*}
where we have used the same consolidation techniques for the norms on $\Delta u$ and $\tilde w$ into norms on $u$ as before.
\end{proof}
We see the variational crimes (arising from the extra $\tilde p_{h}$) account for the sole additional term in the integrals. This cannot be improved without further information on the projections $\Pi_{h}$. Otherwise, for $r=1$, which correspond to piecewise linear discontinuous elements for $2$-forms ($u$), and piecewise quadratic elements for $1$-forms ($\sigma$) with normal continuity (Raviart-Thomas elements), as studied by Thom\'ee, we obtain the estimates he derived (and since the $\tilde p_{h}$ is not there in his case, we have that the extra terms with $u$ do not appear under the integral sign).

\section{A Numerical Example}\label{sec:num-experiments}

In order to actually simulate a solution to the Hodge heat equation, we consider the scalar heat equation on a domain in $M\subseteq \R^{2}$, but now using a mixed method with $2$-forms rather than the functions. We return to the evolution equation for both $\sigma$ and $u$, \eqref{eq:boch-mixedweak-time} above, which we recall here:
\begin{equation}
\begin{tabular}{rllll}
$\aip{\sigma_t,\omega}  + \aip{d\sigma, d\omega}$ & $= \aip{f,d\omega},$ & $\forall~\omega\in V^{k-1},$ & $t\in I$,  \\[2mm]
$\aip{u_{t},\varphi} + \aip{d\sigma,\varphi} + \aip{du,d\varphi}$ & $= \aip{f,\varphi},$ & $\forall~ \varphi\in V^k,$ & $t\in I$, \\[2mm]
$u(0)$ & $= g$.
\end{tabular}
\end{equation}

Given $S_{h} \subseteq V^{k} = H\Omega^{2}(M)$ and $H_{h}\subseteq V^{k-1}= H\Omega^{1}(M)$, we choose bases, and use the semidiscrete equations \eqref{eq:discrete-evolution-system}, which we recall here (setting $U$ to be the coefficients of $u_{h}$ in the basis for $S_{h}$, and $\Sigma$ to be the coefficients of $\sigma_{h}$ in the basis for $H_{h}$)
\begin{equation}
\frac{d}{dt}\begin{pmatrix}
D &-B^{T}\\
0 & A
\end{pmatrix}
\begin{pmatrix}
\Sigma \\ U
\end{pmatrix} = \begin{pmatrix}0 & 0 \\
-B & -K
\end{pmatrix}\begin{pmatrix}
\Sigma \\ U
\end{pmatrix} + \begin{pmatrix} 0\\ F
\end{pmatrix}
\end{equation}
This may be discretized via standard methods for ODEs. For our implementation, we use the backward Euler method.
\begin{figure}[t] 
   \centering
       \begin{subfigure}[t]{.4\linewidth}
    	\centering
    	\includegraphics[width=.8\linewidth]{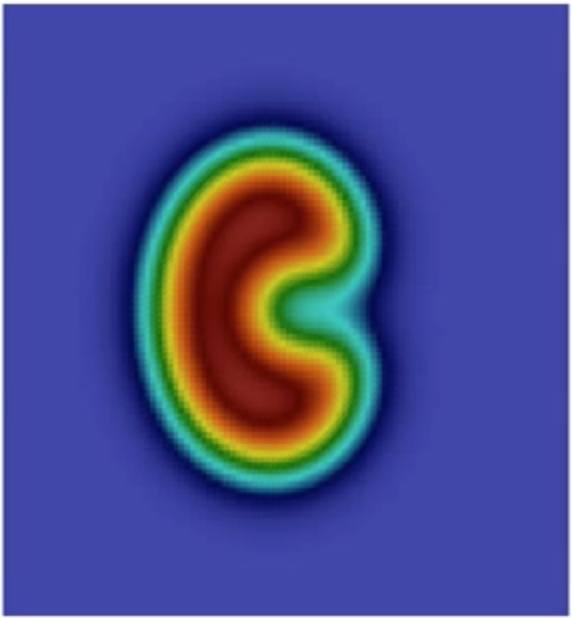} 
         \caption{1 second}
    \end{subfigure}
    \begin{subfigure}[t]{.4\linewidth}
    	\centering
    	\includegraphics[width=.8\linewidth]{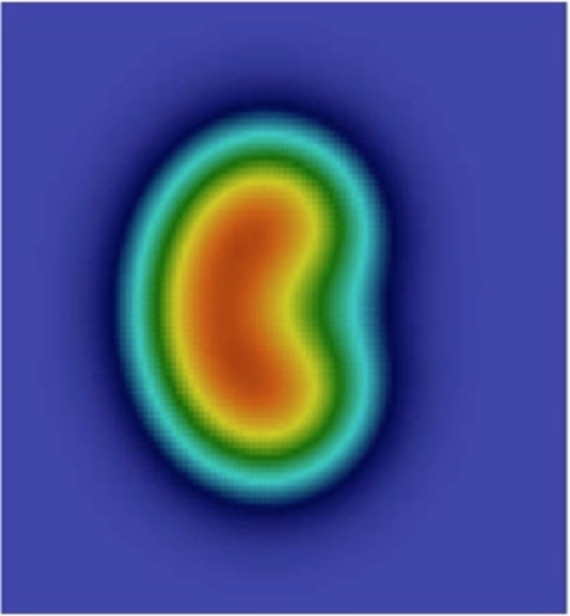} 
         \caption{2 seconds}
    \end{subfigure}
    \begin{subfigure}[t]{.4\linewidth}
    	\centering
    	\includegraphics[width=.8\linewidth]{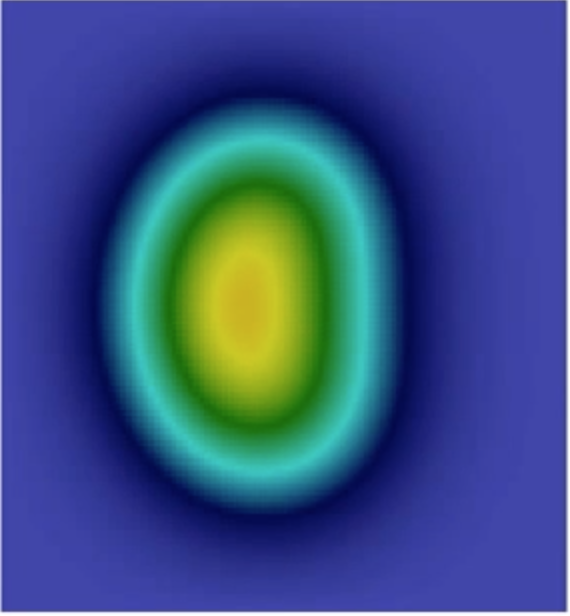} 
         \caption{3 seconds}
    \end{subfigure}
    \begin{subfigure}[t]{.4\linewidth}
    	\centering
    	\includegraphics[width=.8\linewidth]{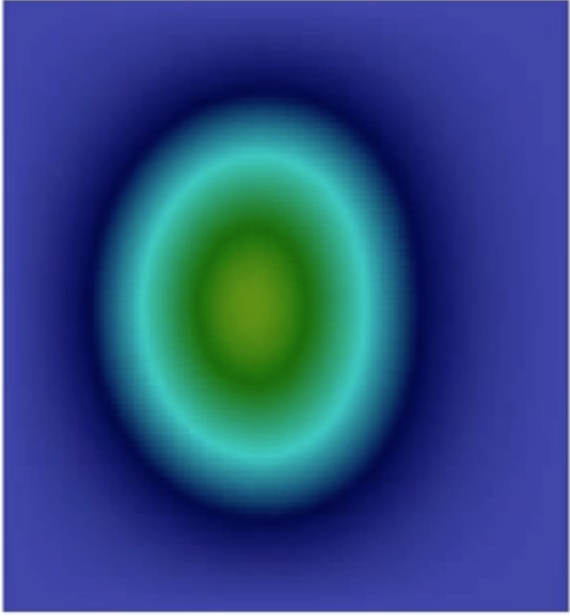} 
         \caption{4 seconds}
    \end{subfigure}
   \caption{Hodge heat equation for $k=2$ in a square modeled as a $100 \times 100$ mesh, using the mixed finite element method given above. Initial data is given as the (discontinuous) characteristic function of a C-shaped set in the square. The timestepping method is given by the backward Euler discretization, with timestep $\Delta t = 5 \times 10^{-5}$.}
   \label{fig:hodgeheat}
\end{figure}
This means we consider sequences $(\Sigma^{n},U^{n})$ in time, and then rewrite the derivative instead as a finite difference, evaluating the vector field portion on the right side at timestep $n+1$, taking $M =\begin{pmatrix}
D &-B^{T}\\
0 & A
\end{pmatrix}$:
\[
\frac{1}{\Delta t} M \left(\begin{pmatrix} \Sigma^{n+1} \\ U^{n+1} \end{pmatrix}-  \begin{pmatrix}\Sigma^{n} \\ U^{n} \end{pmatrix} \right) =  \begin{pmatrix}0 & 0 \\
-B & -K
\end{pmatrix}\begin{pmatrix}
\Sigma^{n+1} \\ U^{n+1}
\end{pmatrix} + \begin{pmatrix} 0\\ F^{n+1}
\end{pmatrix}
\]
or
\[
 \left(M + \Delta t \begin{pmatrix} 0&0\\B&K \end{pmatrix}\right)
\begin{pmatrix}
\Sigma^{n+1}\\
U^{n+1}
\end{pmatrix} = M  \begin{pmatrix}\Sigma^{n} \\ U^{n} \end{pmatrix} + \Delta t \begin{pmatrix} 0\\ F^{n+1}\end{pmatrix}.
\]
We now have written the system as a sparse matrix times the unknown $(\Sigma^{n+1},U^{n+1})$. This allows us to solve the system directly using sparse matrix algorithms without explicitly inverting any matrices, making the iterations efficient. To analyze the error of the approximations, we can combine the above error estimates with the standard error analysis of Euler methods. See Figure \ref{fig:hodgeheat}.


\section{Conclusion and Future Directions}
\label{sec:conc}
We have seen that the abstract theory of Hilbert complexes, as detailed by Arnold, Falk, and Winther~\cite{AFW2010}, and Bochner spaces, as detailed in Gillette and Holst~\cite{GiHo11a} and Arnold and Chen \cite{AC2012}, has been very useful in clarifying the important aspects of elliptic and parabolic equations. The mixed formulation gives great insight into questions of existence, uniqueness, and stability of the numerical methods (linked by the cochain projections $\pi_{h}$). The method of Thom\'ee~\cite{T2006} allows us to leverage the existing theory for elliptic problems to apply to parabolic problems, taking care of the remaining error terms by the use of differential inequalities and Gr\"onwall estimates (in the important error evolution equations \eqref{eq:par-error} above). Incorporating the analysis of variational crimes allow us to carry this theory over to the case of surfaces and their approximations.

We remark on some possible future directions for this work. Some existing surface finite elements for parabolic equations have been studied by Dziuk and Elliott~\cite{DzEl06}, and much other work by Dziuk, Elliott, Deckelnick~\cite{DeDzEl2005,Deckelnick.K;Dziuk.G2003}, which actually treat the case of an evolving surface, and treat a nonlinear equation, the mean curvature flow. Generally speaking, this translates to an additional time dependence for evolving metric coefficients, and a logical place to start is in the Thom\'ee error evolution equations \eqref{eq:par-error}. Nonlinear evolution equations for evolving metrics also suggests the Ricci flow~\cite{Pe03a,ChKn04,CLN06}, instrumental in showing the Poincar\'e conjecture. The challenge there, besides nonlinearity, is that tensor equations do not necessarily fit in the framework for FEEC. On the other hand, the Yamabe flow~\cite{Schoen.R1984}, which solves for a conformal factor for the metric (and is equivalent to the Ricci flow in dimension $2$) suggests an interesting nonlinear scalar evolution equation for which this analysis may be useful.

Gillette, Holst, and Zhu~\cite{GiHo11a} also analyzed hyperbolic equations in this framework, and it would be interesting and useful to analyze methods on surfaces (including the evolving case), as well as taking a more integrated approach in spacetime. This is usually taken care of using the discrete exterior calculus (DEC), the finite-difference counterpart to FEEC to analyze hyperbolic equations \cite{LuMa12}.

\appendix
\begin{appendix}
\end{appendix}

\section{Acknowledgments}
\label{sec:ack}
MH was supported in part by
NSF Awards~1217175, 1262982, and 1318480.
CT was supported in part by
NSF Award~1217175.

\bibliographystyle{abbrv}
\bibliography{../bib/mjh,../bib/akg,../bib/HoSt2010b,../bib/clt,../bib/papers,../bib/books,../bib/library}

\vspace*{0.5cm}

\end{document}